\documentclass[]{article}

\usepackage{amsmath}
\usepackage{amssymb}
\usepackage{amsthm}
\usepackage{authblk}
\usepackage{bm}
\usepackage{booktabs}
\usepackage[letterpaper]{geometry}
\usepackage{hyperref}
\usepackage{mathtools}
\usepackage{siunitx}
\usepackage{xfrac}

\usepackage[
	backend=biber,
	style=numeric-comp,
	isbn=false,
	url=true,
	eprint=false,
	doi=true,
	sorting=none,
]{biblatex}
\title{The Nonparametric Kiefer--Weiss Problem}
\author[1,2]{Michael Fauss}
\author[2]{H.~Vincent Poor}
\author[3]{Abdelhak M.~Zoubir}

\affil[1]{ETS Research Institute} 
\affil[2]{Dept.~of Electrical and Computer Engineering, Princeton University}
\affil[3]{Signal Processing Group, TU Darmstadt}

\theoremstyle{plain}
\newtheorem{theorem}{Theorem}[section]
\newtheorem{lemma}[theorem]{Lemma}
\newtheorem{corollary}[theorem]{Corollary}
\theoremstyle{definition}

\newtheorem{approximation}[theorem]{Approximation}
\DeclareMathOperator*{\argmax}{arg\,max}

\newcommand{\Pb}{\mathbf{P}}

\newcommand{\Ebb}{\mathbb{E}}
\newcommand{\Nbb}{\mathbb{N}}
\newcommand{\Pbb}{\mathbb{P}}
\newcommand{\Rbb}{\mathbb{R}}

\newcommand{\Fcal}{\mathcal{F}}
\newcommand{\Hcal}{\mathcal{H}}
\newcommand{\Ical}{\mathcal{I}}
\newcommand{\Mcal}{\mathcal{M}}
\newcommand{\Xcal}{\mathcal{X}}

\begin{document}

\maketitle

\begin{abstract}
	A nonparametric variant of the Kiefer--Weiss problem is proposed and solved. The objective is to minimize a weighted sum of the error probabilities of a binary sequential test subject to a constraint on its maximum expected sample size. This maximum is taken over all possible probability distributions on the given sequence space. First, it is shown that the nonparametric Kiefer--Weiss problem can be reduced to an optimal stopping problem. Then, the optimal stopping policy is derived under the assumption that at most $k$ uses of randomization are permitted during any run of the test. The solution to the original problem is then obtained by letting $k$ go to infinity. The optimal cost function is shown to be the solution of a nonlinear Bellman equation. The corresponding optimal stopping policy is shown to be based on a two-dimensional test statistic, with one component tracking the likelihood ratio and the other one tracking the expected remaining sample size. Critically, the stopping policy uses randomization to increase the remaining expected sample size for some runs, while stopping early for others. The optimal randomization rule is shown to be determined by a function that maps the likelihood ratio to an integer-valued sample size. Two approximations of this function are proposed that can be evaluated easily in practice. The results are illustrated with two numerical examples of nonparametric Kiefer--Weiss tests, one for a shift in the success probability of a Bernoulli distribution, and one for a shift in the mean of a normal distribution.
\end{abstract}

\section{Introduction}   
\label{sec:introduction} 

A central concern in sequential hypothesis testing is the efficient use of observations: one wishes to minimize the number of samples required to accept one of two or more hypotheses with sufficiently small error probabilities. Wald and Wolfowitz \cite{Wald1948} showed that for two simple hypotheses the sequential probability ratio test (SPRT) is optimal in the sense that it minimizes the expected sample size simultaneously under both hypotheses. However, the sample efficiency of the SPRT is highly sensitive to model mismatch. In fact, when the true distribution differs from the assumed ones, the expected sample size of an SPRT can significantly exceed that of a fixed-sample-size test (FSST) with identical error probabilities---compare, for example, \cite[Figs.~3.4 and 5.1]{Tartakovsky2014}.

Prompted by this observation, Jack Kiefer and Lionel Weiss proposed to design a sequential test that, in addition to meeting the targeted error probabilities under both hypotheses, minimizes the \emph{maximum} expected sample size over a parametric family of distributions \cite{Kiefer1957, Weiss1962}. More formally, Kiefer and Weiss assumed that observations are drawn independently from a parametric distribution $P_\theta$, with $\theta = \theta_0$ and $\theta = \theta_1$ under the respective hypotheses. The Kiefer--Weiss test (KWT) is then defined as the sequential test that admits prescribed Type I and Type II error probabilities and minimizes the maximum expected sample size over $\theta \in \Theta$, where $\Theta$ is a set of feasible parameter values. This problem, in various forms, has received considerable attention in the literature \cite{Dvoretzky1953, Anderson1960, Robbins1972, Lai1973, Lorden1976, Eisenberg1982, Huffman1983, Dragalin1988, Lai1988, Pavlov1991, Augustin2001, Zhitlukhin2013, Novikov2022, Novikov2022a, Novikov2023}.

The general solution of the Kiefer--Weiss problem turned out to be rather elusive, which caused the focus to shift to special cases and approximations. The most well-studied cases are the Kiefer--Weiss test for a shift in the mean parameter of a normal distribution \cite{Anderson1960, Weiss1962, Lai1973, Lorden1976} and the closely related problem of testing the drift parameter of a Wiener process \cite{Dvoretzky1953, Zhitlukhin2013}. Asymptotic, approximate and numeric results exist for more general families of distributions \cite{Robbins1972, Eisenberg1982, Huffman1983, Dragalin1988, Pavlov1991, Augustin2001, Novikov2022, Novikov2022a, Novikov2023}. An up-to-date treatment of the Kiefer--Weiss problem and some of its variations can be found in \cite[Section~5.3]{Tartakovsky2014}.

In this paper, we study a novel variant of the Kiefer--Weiss problem: Our objective is to minimize a weighted sum of the Type I and Type II error probabilities of a binary sequential test subject to a constraint on its maximum expected sample size. However, rather than over a parametric family, this maximum is taken over \emph{all} possible probability distributions on the given sequence space. We therefore refer to the resulting formulation as the \emph{nonparametric Kiefer--Weiss problem}\footnote{We first introduced and investigated this problem in an unpublished preprint \cite{Fauss2020}. The approach used in this paper is fundamentally different and the presented results are significantly stronger and more general. We do not use any results from \cite{Fauss2020}, and do not expect the reader to be familiar with it. However, for the interested reader, \cite{Fauss2020} provides a more detailed discussion and motivation of the Kiefer--Weiss problem.} and the corresponding optimal test as the \emph{nonparametric Kiefer--Weiss test} (NPKWT).

The remainder of the paper is organized as follows: In Section~\ref{sec:notation}, we introduce our notation and assumptions. The core sections of the paper are Section~\ref{sec:problem}, in which the nonparametric Kiefer--Weiss problem is formally defined, and Section~\ref{sec:test}, in which its solution is derived. 
The results are discussed in Section~\ref{sec:discussion}. Two approximations of the optimal stopping policy are proposed in Section~\ref{sec:approximations}. Numerical examples and illustrations are presented in Section~\ref{sec:examples}. Section~\ref{sec:conclusion} concludes the paper and briefly discusses directions for future research. 

\section{Notation and Assumptions} 
\label{sec:notation}							 

Throughout the paper, $\Nbb_{\geq n}$ denotes the set of integers greater than or equal to $n$. Analogously, $\Rbb_{\geq x}$ denotes the set of real numbers greater than or equal to $x$. The set of all sequences with values in $[0, 1]$, sometimes referred to as the Hilbert cube \cite[Section~3.9]{Aliprantis2007}, is denoted by $\Ical^\infty \coloneq [0,1]^\infty$. The indicator function of a set $\mathcal{A}$ is denoted by $\bm{1}_{\mathcal{A}}$. We write $x^* = \argmax_{x \in \Xcal} f(x)$ to indicate that $x^*$ is a maximizer of $f$ over $\Xcal$. Random variables are denoted by uppercase letters, and their realizations by the corresponding lowercase letters. Since it appears frequently in the sequel, we define the function $g \colon \Rbb_{\geq 0} \to [0,1]$ as $g(z) \coloneq \min\{1, z\}$.

Let $\bm{X} = (X_1, X_2, \ldots)$ be a sequence of random variables with values in a measurable space $(\Xcal, \Fcal)$. The truncated sequence $(X_1, \ldots, X_t)$ is denoted by $\bm{X}_t$. The joint distribution of $\bm{X}$ is denoted by $\Pb \in \Mcal(\Xcal^\infty)$, where $\Mcal(\Xcal^\infty)$ denotes the set of all probability measures on the sequence space $\Xcal^\infty$. We assume that two simple hypotheses about $\Pb$ are given:
\begin{equation*}
	\begin{aligned}
		\Hcal_0 &\colon \bm{X} \text{ i.i.d.\ according to } P_0, \\ 
		\Hcal_1 &\colon \bm{X} \text{ i.i.d.\ according to } P_1.
	\end{aligned}
\end{equation*}
In words, we seek to test whether the elements of $\bm{X}$ are independent and identically distributed (i.i.d.) according to $P_0$ or to $P_1$. For the sake of exposition, it is further assumed that $P_1$ is absolutely continuous with respect to $P_0$ and that the Kullback--Leibler divergence $D_\text{KL}(P_0 \Vert P_1)$ is non-zero and finite. The densities of $P_0$ and $P_1$ with respect to a suitable background measure are denoted by $p_0$ and $p_1$, respectively. Possible relaxations of these assumptions are discussed in Section~\ref{sec:discussion}.

At every time $t$, a sequential test needs to make a decision to stop or continue. Let $\bm{S} = (S_0, S_1, \ldots)$ denote a sequence of binary random variables corresponding to this decision, that is, the sequential test stops at time $t$ if $S_t = 1$ and continues if $S_t = 0$. The corresponding stopping time is defined as
\begin{equation*}
	T = \min\{\, t \geq 0 : S_t = 1 \,\}.
\end{equation*}
Upon stopping, the test outputs a binary decision $D$, where $D = 0$ indicates acceptance of $\Hcal_0$, and $D = 1$ indicates acceptance of $\Hcal_1$. 

Let $\Pbb$ be the joint distribution of $(\bm{X}, \bm{S}, D)$, and let $\Ebb$ denote the corresponding expectation operator. We define the conditional expectations:
\begin{equation}
	\psi_t = \psi_t(\bm{x}_t) \coloneq \Ebb\bigl[ S_t \,|\, \bm{X}_t = \bm{x}_t \bigr]
	\label{eq:psi_t}
\end{equation}
and
\begin{equation}
	\delta_t = \delta_t(\bm{x}_t) \coloneq \Ebb\bigl[ D \,|\, T = t, \bm{X}_t = \bm{x}_t \bigr].
	\label{eq:delta_t}
\end{equation}
The sequence $\psi = (\psi_0, \psi_1, \ldots) \in \Ical^\infty$ is referred to as a \emph{stopping policy}, and the sequence $\delta = (\delta_0, \delta_1, \ldots) \in \Ical^\infty$ as a \emph{decision policy}. The pair $(\psi, \delta)$ is called a \emph{testing policy} and completely specifies a sequential test. 

For notational convenience, we define the shorthand
\begin{equation}
	\phi_t = \phi_t(\bm{x}_t) \coloneq \Pbb\bigl[T = t \mid \bm{X}_t = \bm{x}_t \bigr]
	= \psi_t(\bm{x}_t)\prod_{n=0}^{t-1} \bigl(1 - \psi_n(\bm{x}_n)\bigr),
	\label{eq:phi_t}
\end{equation}
which denotes the probability of stopping at time $t$ given observations $\bm{x}_t$. We further define
\begin{equation}
	\tau(\bm{x}) \coloneq T \mid (\bm{X} = \bm{x})
	\label{eq:tau}
\end{equation}
to represent the stopping time associated with a given observation sequence $\bm{x}$. Note that $\tau(\bm{x})$ is deterministic for deterministic stopping policies but random in general.

For notational clarity, we make explicit the dependence of $\Pbb$ on the sample distribution, the stopping policy, and the decision policy. Specifically, we write $\Pbb = \Pbb_{\Pb, \psi, \delta}$ and $\Ebb = \Ebb_{\Pb,\psi,\delta}$ for the probability distribution and expectation, respectively. Under hypothesis $\Hcal_i$, $i \in \{0,1\}$, we write $\Pbb_{i,\psi,\delta}$ and $\Ebb_{i,\psi,\delta}$. Whenever a probability or expectation does not depend on $\Pb$, $\psi$, or $\delta$, the corresponding subscripts are omitted. For example, $\Pbb_{\Pb, \psi, \delta}\bigl[ T = 0 \bigr] = \Pbb_{\psi}\bigl[ T = 0 \bigr]$.

Finally, the Type~I and Type~II error probabilities of a sequential test of $\Hcal_1$ against $\Hcal_0$ using policy $(\psi,\delta)$ are written as
\begin{equation*}
	\alpha(\psi,\delta) \coloneq \Ebb_{0,\psi,\delta}[D]
	\quad \text{and} \quad
	\beta(\psi,\delta) \coloneq \Ebb_{1,\psi,\delta}[1-D].
\end{equation*}
We are now ready to formulate the nonparametric Kiefer--Weiss problem.

\section{The Nonparametric Kiefer--Weiss Problem} 
\label{sec:problem}																

Let $\Ical_c^\infty$ denote the set of stopping policies for which the expected sample size of the underlying sequential test is bounded by some constant $c \geq 0$ for all possible distributions of $\bm{X}$:
\begin{equation}
	\Ical_c^\infty \coloneq \bigl\{\psi \in I^\infty \colon \Ebb_{\Pb, \psi}\bigl[\, T \,\bigr] \leq c \quad \forall \, \Pb \in \Mcal(\Xcal^\infty) \bigr\}.
	\label{eq:stopping_policy_prob}
\end{equation}
We define the nonparametric Kiefer--Weiss problem as follows:
\begin{equation}
	\inf_{\psi \in \Ical_c^\infty} \; \inf_{\delta \in \Ical^\infty} \; \alpha(\delta, \psi) + z \, \beta(\delta, \psi),
	\label{eq:nonpa_kiwei}
\end{equation}
where $z, c \geq 0$. In words, we seek to design a policy for a sequential test between two simple hypotheses that minimizes the weighted sum of its error probabilities, subject to the constraint that its expected sample size does not exceed $c$ for all possible distributions of $\bm{X}$.

In the remainder of this section, some useful simplifications of \eqref{eq:nonpa_kiwei} are derived. First, we show that \eqref{eq:stopping_policy_prob} can equivalently be defined via a pointwise constraint on the expected conditional sample size, $\Ebb_\psi[\tau(\bm{x})]$.

\begin{theorem}
	The set of stopping policies $\Ical_c^\infty$ in \eqref{eq:stopping_policy_prob} can equivalently be defined as
	\begin{equation}
		\Ical_c^\infty = \bigl\{\psi \in I^\infty \colon \Ebb_\psi\bigl[ \tau(\bm{x}) \bigr] \leq c \quad \forall \, \bm{x} \in \Xcal^\infty \bigr\}.
		\label{eq:stopping_policy_pointwise}
	\end{equation}
	\label{th:equivalence}
\end{theorem}
Theorem~\ref{th:equivalence} is proven in Appendix~\ref{apx:equivalence}. 

Next, it is shown that the optimal \emph{testing} problem can be reduced to an optimal \emph{stopping} problem by minimizing out the decision policy. Variations of this result are well-known in sequential analysis. It is stated and proven here for completeness.
\begin{theorem}
	\label{th:reduction}
	The decision rule $\delta^*$ that solves the inner minimization in \eqref{eq:nonpa_kiwei} is a likelihood ratio test of the form
	\begin{equation}
		\delta_t^*(\bm{x}_t) \begin{dcases}
			= 0, & z \prod_{n=1}^t \frac{p_1(x_n)}{p_0(x_n)} < 1, \\
			\in [0, 1], & z \prod_{n=1}^t \frac{p_1(x_n)}{p_0(x_n)} = 1, \\
			= 1, & z \prod_{n=1}^t \frac{p_1(x_n)}{p_0(x_n)} > 1. \\
		\end{dcases}
		\label{eq:delta_opt}
	\end{equation}
	Moreover, it holds that
	\begin{equation}
		\alpha(\psi, \delta^*) + z \, \beta(\psi, \delta^*)
		= \Ebb_{0, \psi}\Biggl[ g\biggl(z \prod_{t=1}^T \frac{p_1(X_t)}{p_0(X_t)} \biggr) \Biggr].
		\label{eq:cost_delta_opt}
	\end{equation}
\end{theorem}
Theorem~\ref{th:reduction} is proven in Appendix~\ref{apx:reduction} 

The nonparametric Kiefer--Weiss problem can now be written as
\begin{equation}
	\inf_{\psi \in \Ical_c^\infty} \; \alpha(\psi) + z \, \beta(\psi),
	\label{eq:nonpa_kiwei_stopping}
\end{equation}
where, in a slight abuse of notation, we defined $\alpha(\psi) \coloneq \alpha(\psi, \delta^*)$ and $\beta(\psi) \coloneq \beta(\psi, \delta^*)$.

We conclude this section by showing existence of an optimal stopping policy.
\begin{lemma}
	There exists a stopping policy that attains the infimum in \eqref{eq:nonpa_kiwei_stopping}.
	\label{lm:existence}
\end{lemma}
Lemma~\ref{lm:existence} is proven in Appendix~\ref{apx:existence}.

\section{The Nonparametric Kiefer--Weiss Test} 
\label{sec:test}																			 

In this section, we derive the stopping policy that solves \eqref{eq:nonpa_kiwei_stopping}. To this end, a recurrence relation is established that leads to a characterization of the optimal stopping policy as the solution of a nonlinear integral equation of the Bellman type. This equation is derived by considering policies that are restricted in terms of how many \emph{randomized} stopping decision can be made during the test. We refer to this as the limited-randomness case. The solution to the original problem is then obtained by letting the number of allowed randomization uses go to infinity.

\subsection{Limited Randomness} %

Let $\Ical_{c,k}^\infty$ denote the set of policies that are feasible in the sense of \eqref{eq:stopping_policy_pointwise}, but whose use of randomization is limited to $k \in \Nbb_{\geq 0}$ instances:
\begin{equation*}
	\Ical_{c,k}^\infty \coloneq \biggl\{\psi \in \Ical_c^\infty : \sum_{t = 0}^T \bm{1}_{(0,1)}(\psi_t) \leq k \biggr\}.
\end{equation*}
We define the family of functions $\rho_k \colon \Rbb_{\geq 0}^2 \to [0,1]$ as
\begin{equation}
	\rho_k(z, c) \coloneq \inf_{\psi \in \Ical_{c,k}^\infty} \; \alpha(\psi) + z \, \beta(\psi).
	\label{eq:rho_k_def}
\end{equation}
Limiting the use of randomized stopping decisions causes the space $\Ical_{c,k}^\infty$ to be non-compact in general. Therefore, we cannot assume the existence of an optimal stopping policy. However, the optimal policy can be derived in a constructive manner. Before doing so, some useful properties of $\rho_k$ that follow directly from the definition in \eqref{eq:rho_k_def} are stated.
\begin{lemma}
	The function $\rho_k(z, c)$ defined in \eqref{eq:rho_k_def} is
	\begin{enumerate}
		\item non-decreasing and concave in $z$;
		\item non-increasing in $c$ and $k$; and
		\item upper-bounded by $g(z)$.
	\end{enumerate}
	\label{lm:rho_k_properties}
\end{lemma}
Lemma~\ref{lm:rho_k_properties} is proven in Appendix~\ref{apx:rho_k_properties}. 

Now, assume that $k = 0$, meaning that the stopping policy is constrained to be deterministic. In this case, the constraint in \eqref{eq:stopping_policy_pointwise} becomes
\begin{equation*}
	\tau(\bm{x}) \leq c \quad \forall \bm{x} \in \Xcal^\infty,
\end{equation*}
that is, the underlying sequential test is \emph{truncated} after at most $c$ samples. In the next lemma, it is shown that for deterministic stopping rules the nonparametric Kiefer--Weiss test reduces to a fixed-sample size test.
\begin{lemma}
	For $k = 0$, it holds that
	\begin{equation*}
		\rho_0(z, c) = \Ebb_0\Biggl[ g\biggl(z \prod_{t=1}^{\lfloor c \rfloor} \frac{p_1(X_t)}{p_0(X_t)} \biggr) \Biggr].
	\end{equation*}
	Thus, for all $z \geq 0$, the function $\rho_0(z, \bullet)$ is piecewise constant with breakpoints in $\Nbb_{\geq 0}$.
	\label{lm:rho_0}
\end{lemma}
\begin{proof}
	For stopping rules in $\Ical_{c, 0}^\infty$, it holds that $T \leq c$ since $T > c$ implies that $\Ebb_\psi[\tau(\bm{x})] = \tau(\bm{x}) > c$ for at least one $\bm{x} \in \Xcal^\infty$. From the concavity of $g$ together with Jensen's inequality, it further follows that the cost function on the right-hand side of \eqref{eq:cost_delta_opt} is non-increasing in $T$. Consequently, the minimum is attained at the largest feasible integer, that is, $T = \lfloor c \rfloor$. 
\end{proof} 
Next, we provide a recursive expression for $\rho_k$ when $k \geq 1$.
\begin{lemma}
	Define the function $r_k \colon \Rbb_{\geq 0} \times \Rbb_{\geq 1} \to [0, 1]$ as
	\begin{align}
		r_k(z, b) \coloneq \frac{1}{b}\left( g(z) - \Ebb_0\biggl[\rho_k\biggl(z \frac{p_1(X)}{p_0(X)}, b - 1\biggr) \biggr] \right).
		\label{eq:r_def}
	\end{align}
	For all $k \geq 1$ it holds that
	\begin{equation}
		\rho_k(z, c) = \begin{dcases}
			g(z) - c \, \max_{b \geq 1} \, r_{k-1}(z, b), & c < 1, \\[2pt]
			\min\biggl\{\Ebb_0\biggl[\rho_k\biggl( z \frac{p_1(X)}{p_0(X)}, c - 1\biggr)\biggr] \,,\, g(z) - c \, \max_{b \geq c} \, r_{k-1}(z, b) \biggr\}, & c \geq 1.
		\end{dcases}
		\label{eq:rho_k_recursive}
	\end{equation}
	\label{lm:rho_k_recursive}
\end{lemma}
Lemma~\ref{lm:rho_k_recursive} is proven in Appendix~\ref{apx:rho_k_recursive}. It completely specifies the optimal nonparametric Kiefer--Weiss test with at most $k$ randomized stopping decisions. More specifically, given $\rho_{k-1}$ for some $k \geq 0$, $\rho_k(z, \bullet)$ can be calculated on the interval $[0, 1)$ via the first case in \eqref{eq:rho_k_recursive}. Given $\rho_k(z, \bullet)$ on $[0, 1)$, it can be evaluated on $[1, 2)$ via the second case in \eqref{eq:rho_k_recursive}. Continuing in this manner, $\rho_k(z, \bullet)$ can be pieced together on $\Rbb_{\geq 0}$ and can in turn be used to calculate $\rho_{k+1}$. The corresponding optimal stopping rules can be obtained from the arguments that maximize/minimize the respective terms in \eqref{eq:rho_k_recursive}. This result is made formal in the next theorem. To simplify its statement, we define the following condition:
\begin{equation}
	\Ebb_0\Bigl[\rho_k\Bigl( z \frac{p_1(X)}{p_0(X)}, c - 1\Bigr)\Bigr] \leq g(z) - c \, \max_{b \geq c} \, r_{k-1}(z, b).
	\label{eq:cont_condition}
\end{equation}
This condition is true if the minimum in the case $c \geq 1$ in \eqref{eq:rho_k_recursive} is attained by its first argument.

\begin{theorem}
	Let $(z_t, c_t, k_t) \in \Rbb_{\geq 0}^2 \times \Nbb_{\geq 0}$ be a sequential test statistic with initial value $(z_0, c_0, k_0) = (z, c, k)$ and update rule
	\begin{align*}
		z_{t+1} &= z_t \, \frac{p_1(x_{t+1})}{p_0(x_{t+1})},\\
		c_{t+1} &= b_{k_t}^*(z_t, c_t) - 1,\\
		k_{t-1} &= \begin{cases}
			k_t, & b_{k_t}^*(z_t, c_t) = c_t, \\
			k_t -1, & b_{k_t}^*(z_t, c_t) > c_t,
		\end{cases}
	\end{align*}
	where $b_k^*$ is given by
	\begin{equation}
		b_k^*(z, c) = \begin{dcases}
			c, & (k = 0) \ \text{or} \ (c \geq 1 \ \text{and} \ \eqref{eq:cont_condition} \ \text{is true}), \\
			\argmax_{b \geq 1} \, r_{k-1}(z, b), &  \text{otherwise}.
			\label{eq:b_opt_k}
		\end{dcases}
	\end{equation}
	The stopping rule
	\begin{align}
		\psi_{k,t}^* = 1 - \frac{c_t}{b_{k_t}^*(z_t, c_t)}
		\label{eq:stopping_rule_k}
	\end{align}
	is optimal in the sense of \eqref{eq:rho_k_def}.
	\label{th:optimal_stopping_k}
\end{theorem}

\begin{proof}
	The theorem follows directly from the proof of Lemma~\ref{lm:rho_k_recursive} in Appendix~\ref{apx:rho_k_recursive}. There, it is shown that the optimal initial stopping rule is deterministic if either $k = 0$ or $c \geq 1$ and the condition in \eqref{eq:cont_condition} is true. In all other cases, the optimal stopping rule is randomized and of the form $1 - \frac{c}{b}$. Optimizing over $b$ yields \eqref{eq:stopping_rule_k} in the theorem. Given that the test did not stop, it continues with an optimal policy for the new parameters $z \leftarrow z \frac{p_1(x_1)}{p_0(x_1)}$, $c \leftarrow b_k^*(z, c) - 1$ and $k \leftarrow k - \bm{1}_{(0,1)}(\psi_0^*)$, where $\psi_0^* \in (0,1)$ if and only if $b_k^*(z, c) > c$. Proceeding in this manner gives the statement in the theorem.
\end{proof}

The three quantities tracked by the test statistic in Theorem~\ref{th:optimal_stopping_k} are the likelihood ratio, $z_t$, the maximum expected remaining sample size, $c_t$, and the number of remaining randomization uses, $k_t$. The underlying idea is that randomized stopping decisions allow the test to increase the number of samples in some runs at the expense of stopping early in others. By carefully choosing when and how to randomize, this policy can achieve lower error probabilities than a truncated SPRT or an FSST. A more detailed discussion of the optimal policy is deferred to Section~\ref{sec:discussion}.  

Although its optimal stopping policy can be calculated explicitly, the limited-randomness case is somewhat unsatisfying, both from a theoretical and a practical perspective. The constraint on the use of randomization is arguably artificial and unlikely to arise in practice. Moreover, recursive calculation of $\rho_k$ via \eqref{eq:rho_k_recursive} is computationally expensive and error prone. In general, $\rho_k$ is non-convex and discontinuous, making it hard to approximate it numerically and to solve the maximization problems in \eqref{eq:rho_k_recursive}---numerical examples will be shown in Section~\ref{sec:discussion}. Interestingly, these complications disappear as $k \to \infty$. This limit is investigated in the next section.

\subsection{Unlimited Randomness} %

The unlimited randomness case can be recovered from the limited randomness case by letting $k \to \infty$. In the next lemma, we show that $\rho_k$ converges to the minimum cost of the original problem in \eqref{eq:nonpa_kiwei_stopping}.
\begin{lemma}
	The sequence $( \rho_k )_{k \geq 0}$ converges pointwise to the minimum in \eqref{eq:nonpa_kiwei_stopping}:
	\begin{equation*}
		\lim_{k \to \infty} \rho_k(z, c) = \rho(z, c) = \min_{\psi \in \Ical_c^\infty} \; \alpha(\psi) + z \, \beta(\psi)
	\end{equation*}
	\label{lm:rho_limit}
\end{lemma}

\begin{proof}
	Let $z, c \geq 0$ be given. Since the sequence $(\rho_k(z,c))_{k \geq 0}$ is non-increasing and bounded from below, it converges to a unique limit $\rho(z,c)$. Moreover, since $\Ical_{c,k}^\infty \uparrow \Ical_c^\infty$, this limit is the infimum in \eqref{eq:nonpa_kiwei_stopping}. The existence of the corresponding minimizer was shown in Lemma~\ref{lm:existence}. This argument holds for all $z, c \geq 0$.
\end{proof}

Next, we characterize $\rho$ in terms of an integral equation of the Bellman type.
\begin{lemma}
	The function $\rho$ defined in Lemma~\eqref{lm:rho_limit} is the unique solution of the integral equation
	\begin{equation}
		\frac{g(z) - \rho(z, c)}{c} = \max_{b \geq \max\{1,c\}} \frac{g(z) - \Ebb_0\bigl[\rho\bigl(z\frac{p_1(X)}{p_0(X)}, b - 1\bigr)\bigr]}{b}
		\label{eq:rho_bellman}
	\end{equation}
	with boundary condition $\rho(z, 0) = g(z)$ on $\Rbb_{\geq 0}^2$.
	\label{lm:integral_equation}
\end{lemma}
Lemma~\ref{lm:integral_equation} is proven in Appendix~\ref{apx:integral_equation}.

Comparing \eqref{eq:rho_bellman} to \eqref{eq:rho_k_recursive}, it appears that $\rho$ is smoother than $\rho_k$ in the sense that it requires fewer case-by-case comparisons and is not explicitly derived from the discontinuous function $\rho_0$. However, given its derivation, one would still expect $\rho$ to show remnants of discontinuity. This intuition is formalized in the next lemma.

\begin{lemma}
	The function $\rho$ that solves \eqref{eq:rho_bellman} has all properties listed in Lemma~\ref{lm:rho_k_properties}. Moreover, for all $z \geq 0$ the function $\rho(z, \bullet)$ is convex and piecewise linear with breakpoints in $\Nbb_{\geq 0}$.
	\label{lm:rho_properties}
\end{lemma}

Lemma~\ref{lm:rho_properties} is proven in Appendix~\ref{apx:rho_properties}. It is key to simplifying the optimal stopping rule and allows us to eliminate most of the complications that occurred in the limited randomness case. First, since a convex and piecewise linear function is completely specified by its breakpoints, we can restrict the domain of the second argument of $\rho$ to $\Nbb_{\geq 0}$ without loss of generality.

\begin{corollary}
	The integral equation 
	\begin{equation}
		\frac{g(z) - \rho(z, n)}{n} = \sup_{m \geq n} \frac{g(z) - \Ebb_0\bigl[\rho\bigl(z\frac{p_1(X)}{p_0(X)}, m - 1\bigr)\bigr]}{m}
		\label{eq:rho_bellman_discrete}
	\end{equation}
	with boundary condition $\rho(z, 0) = g(z)$ has a unique solution on $\Rbb_{\geq 0} \times \Nbb_{\geq 0}$.
\end{corollary}

\begin{proof}
	The corollary is a direct consequence of Lemma~\ref{lm:integral_equation} and Lemma~ \ref{lm:rho_properties}.
\end{proof}

We can now state the optimal stopping policy of the nonparametric Kiefer--Weiss test. This constitutes the main result of the paper.
\begin{theorem}
	Let $(z_t, c_t) \in \Rbb_{\geq 0}^2$ be a sequential test statistic with initial value $(z_0, c_0) = (z, c)$ and update rule
	\begin{align*}
		z_{t+1} &= z_t \, \frac{p_1(x_{t+1})}{p_0(x_{t+1})} \\ 
		c_{t+1} &= \max\{ c_t \,,\, m^*(z_t) \} - 1,
	\end{align*}
	where 
	\begin{equation}
		m^*(z) = \argmax_{m \geq 1} \frac{g(z) - \Ebb_0\bigl[\rho\bigl(z\frac{p_1(X)}{p_0(X)}, m - 1\bigr)\bigr]}{m},
		\label{eq:m_opt}
	\end{equation}
	and $\rho$ solves \eqref{eq:rho_bellman_discrete}. The stopping rule
	\begin{equation}
		\psi_t^* = 1 - \frac{c_t}{\max\{ c_t \,,\, m^*(z_t) \}}
		\label{eq:stopping_rule}
	\end{equation}
	is optimal in the sense of \eqref{eq:nonpa_kiwei_stopping}.
	\label{th:optimal_stopping}
\end{theorem}
Theorem~\ref{th:optimal_stopping} is proven in Appendix~\ref{apx:optimal_stopping}.

\section{Discussion}	 
\label{sec:discussion} 

By inspection of Theorem~\ref{th:optimal_stopping}, the NPKWT uses a two-dimensional test statistic: $z_t$ tracks the likelihood ratio, and $c_t$ tracks the expected number of remaining samples. The evolution of $c_t$, and in turn the stopping policy of the test, is completely specified by the function $m^*$, which maps a (real-valued) likelihood ratio to an (integer-valued) sample size. At time $t$, the test stops with certainty if and only if $c_t = 0$, and it continues with certainty if and only if $c_t \geq m^*(z_t)$. In all other cases, the stopping decision is randomized and the test continues with probability $\frac{c_t}{m^*(z_t)}$.

The NPKWT can be interpreted as a conventional SPRT augmented with a ``sample budget,'' $c_t$. As long as this budget is sufficiently large, the test continues with certainty and decrements $c_t$ with every observation. In this regime, the NPKWT effectively behaves like an FSST. However, its behavior changes when $c_t$ is small enough that the test outcome is effectively determined, meaning the probability of the remaining samples changing the decision is close to zero. In this case, the NPKWT either stops immediately and saves the remaining samples for another run, or continues with a sample budget that is sufficiently large to potentially affect its outcome. The function $m^*$ makes this notion precise: Given the likelihood ratio $z_t$, continuing the test is only worthwhile if the available sample budget is at least $m^*(z_t)$.

The stopping policy of the NPKWT will be illustrated with concrete, numerical examples in Section~\ref{sec:examples}. Before concluding this section, some additional remarks are useful:

\begin{itemize}
	\item The optimal stopping policy, $\psi^*$, satisfies $\Ebb_{\psi^*}[\tau(\bm{x})] = c$ for all $\bm{x} \in \Xcal^\infty$. That is, the NPKWT has a constant expected sample size for all possible sequences $\bm{x} \in \Xcal^\infty$. In turn, \emph{every} distribution is least favorable. This ``equalization of cost'' is a well-known characteristic of minimiax procedures; see, for example, \cite[Chapter~5, Corollaries 1.5 and 1.6]{Lehmann1998}. Interestingly, by construction, the equalization property not only holds for $m^*$, but for all functions $\Rbb_{\geq 0} \to \Nbb_{\geq 1}$. Suboptimal choices will increase the error probabilities, but will not violate the sample size constraint. In this sense, the NPKWT is robust against deviations from the optimal policy. 
	
	\item In contrast to most sequential tests, including the parametric KWT, the stopping rule of the NPKWT is not threshold-based. Loosely speaking, the NPKWT stops because of ``insufficient remaining samples,'' not because of ``sufficient evidence.''
	
	\item The NPKWT is \emph{untruncated} in the sense that its stopping time is generally supported on $\Nbb_{\geq 0}$. Arguably, this is an unexpected property for at least two reasons: First, given the underlying problem formulation, one might expect the NPKWT to be truncated in order to avoid sequences with unbounded stopping time. This intuition is true for deterministic stopping policies, for which the existence of such a sequence indeed implies an unbounded expected stopping time. For randomized policies, however, each sequence is associated with a stopping time \emph{distribution} whose expectation can be bounded, despite its support being unbounded. Second, it is well-known that the parametric KWT is typically truncated \cite[Section~5.3]{Tartakovsky2014}. Since the NPKWT is a restriction of the KWT, it seems natural that it should inherit this property. However, the significantly stricter constraint on the expected sample size changes the character of the test. As discussed above, the NPKWT is no longer a threshold-based sequential test, so that intuitions learned from the latter do not necessarily carry over. 
	
	\item When the NPKWT continues after a randomized stopping decision, its expected remaining sample size becomes integer-valued since $m^*$ maps to $\Nbb_{\geq 1}$. In other words, $c_t$ eventually snaps to the integer grid. Moreover, since $\rho(z, \bullet)$ is affine between integers, an optimal test for a non-integer-valued $c$ can be implemented by mixing the optimal policies for $\lfloor c \rfloor$ and $\lceil c \rceil$. In this sense, policies for non-integer-valued $c$ are redundant. This observation also supports the interpretation of the NPKWT as a relaxed FSST.
	
	\item In practice, it is typically easier to work with logarithmic instead of linear likelihood ratios. In this case, instead of solving \eqref{eq:rho_bellman_discrete}, one can directly solve
	\begin{equation}
		\frac{g(a^\lambda) - \varrho(\lambda, n)}{n} = \sup_{m \geq n} \frac{g(a^\lambda) - \Ebb_0\bigl[\varrho\bigl(\lambda + \ell_a(X), m - 1\bigr)\bigr]}{m}
		\label{eq:rho_bellman_discrete_llr}
	\end{equation}
	with boundary condition $\varrho(\lambda, 0) = g(a^\lambda)$ on $\Rbb \times \Nbb_{\geq 0}$, where $a$ denotes the base of the logarithm and
	\begin{equation}
		\ell_a(X) \coloneq \log_a \frac{p_1(X)}{p_0(X)}.
		\label{eq:llr}
	\end{equation}
	To obtain the counterpart of $m^*$ in the log-domain, $\varrho$ can be substituted for $\rho$ in \eqref{eq:m_opt}.
	
	\item In \eqref{eq:nonpa_kiwei}, we formulate the nonparametric Kiefer--Weiss problem as an error probability minimization under a constraint on the maximum expected sample size. An alternative formulation, which is arguably more in the spirit of sequential analysis, is to minimize the maximum expected sample size under constraints on the error probabilities:
	\begin{equation}
		\inf_{\psi, \delta \in \Ical^\infty} \; \sup_{\Pb \in \Mcal(\Xcal^\infty)} \; \Ebb_{\Pb, \psi}\bigl[ T \bigr] \quad \text{s.t.} \quad \alpha(\delta, \psi) \leq \overline{\alpha}, \quad \beta(\delta, \psi) \leq \overline{\beta}.
		\label{eq:nonpa_kiwei_alternative}
	\end{equation}
	Using Lagrange duality, it is not hard to show that \eqref{eq:nonpa_kiwei_alternative} can be solved by finding a pair $(z, c)$ such that the corresponding optimal policy, $(\delta^*, \psi^*)$, saturates the error probability constraints. In other words, $(\delta^*, \psi^*)$ solves \eqref{eq:nonpa_kiwei_alternative} if $\alpha(\delta^*, \psi^*) = \overline{\alpha}$ and $\beta(\delta^*, \psi^*) = \overline{\beta}$. However, while any decision rule of the form \eqref{eq:delta_opt} solves \eqref{eq:nonpa_kiwei}, the solution of \eqref{eq:nonpa_kiwei_alternative} can require a particular randomization.
	
	\item The assumptions that $P_1$ is absolutely continuous with respect to $P_0$ and that $D_\text{KL}(P_0 \lVert P_1)$ is finite can be relaxed. In fact, the latter assumption is not needed for the derivation of the main result but is required for the approximations proposed in the next section. The first assumption can be relaxed by defining $\frac{p_1}{p_0} \coloneq \infty$ on $\{ x \in \Xcal \colon p_0(x) = 0, p_1(x) > 0\}$ and $\rho(\infty, c) \coloneq 1$.
\end{itemize}

\section{Approximate Stopping Policies} 
\label{sec:approximations}							

In this section, we propose two approximations of the function $m^*$ in \eqref{eq:m_opt}. Both are obtained by considering the single-randomization case, $k = 1$. In this case, the optimal remaining sample size after a randomized stopping decision is characterized by $\rho_0$, which is the cost of a fixed-sample size test and can be closely approximated using standard techniques.

It follows from Theorem~\ref{th:reduction} and Lemma~\ref{lm:rho_0} that for a given pair $z, c \geq 0$ the optimal test with deterministic stopping policy is a likelihood ratio test of size $\lfloor c \rfloor$ with decision threshold $z^{-1}$. Using a normal approximation for the log-likelihood ratio distribution one obtains
\begin{align}
	\rho_0(z, c) &= \Pb_0\Biggl[ \prod_{t=1}^{\lfloor c \rfloor} \frac{p_1(X_t)}{p_0(X_t)} > \frac{1}{z} \Biggr] + z \, \Pb_1\Biggl[ \prod_{t=1}^{\lfloor c \rfloor} \frac{p_1(X_t)}{p_0(X_t)} \leq \frac{1}{z} \Biggr] \label{eq:rho0_explicit} \\
	&= \Pb_0\Biggl[ \sum_{t=1}^{\lfloor c \rfloor} \log \frac{p_1(X_t)}{p_0(X_t)} > -\log z \Biggr] + z \, \Pb_1\Biggl[ \sum_{t=1}^{\lfloor c \rfloor} \log \frac{p_1(X_t)}{p_0(X_t)} \leq - \log z \Biggr] \notag \\
	&\approx 1 - \Phi\biggl( -\frac{\log z + \mu_0 \lfloor c \rfloor}{\sigma_0 \sqrt{\lfloor c \rfloor}} \biggr) + z \, \Phi\biggl( -\frac{\log z + \mu_1 \lfloor c \rfloor}{\sigma_1 \sqrt{\lfloor c \rfloor}} \biggr) \notag \\
	&\approx \Phi\biggl(\frac{\log z + \mu_0 c}{\sigma_0 \sqrt{c}} \biggr) + z \, \Phi\biggl( -\frac{\log z + \mu_1 c}{\sigma_1 \sqrt{c}} \biggr), \label{eq:rho0_approx}
\end{align}
where $\Phi$ denotes the cumulative distribution function (CDF) of the standard normal distribution, and $\mu_i, \sigma_i^2$, $i \in \{0,1\}$, denote the mean and variance of the log-likelihood ratio in \eqref{eq:llr} under $\Hcal_i$, respectively. For brevity, we omit the dependence on the base $a$ in this section. Note that in \eqref{eq:rho0_explicit} the underlying test decides for $\Hcal_0$ if the test statistic hits the threshold exactly. This assignment is arbitrary and without loss of optimality---compare Theorem~\ref{th:reduction}. Finally, in the last step, the flooring operation is dropped to simplify the expression.

Substituting the right-hand side of \eqref{eq:rho0_approx} for $\rho_0$ in \eqref{eq:r_def} yields an approximation of $r_0$ and in turn of $b_1^*$. This is our fist approximation of $m^*$.
\begin{approximation}[Approximate Upper Bound]
	\begin{equation}
		m^*(z) \lesssim \argmax_{b \geq 1} \, \frac{1}{b}\left( g(z) - \Phi\biggl( \frac{\log z + \mu_0 b}{\sigma_0 \sqrt{b}} \biggr) - z \, \Phi\biggl( -\frac{\log z + \mu_1 b}{\sigma_1 \sqrt{b}} \biggr) \right).
		\label{eq:approx_ub}
	\end{equation}
\end{approximation}
The claim that the right-hand side of \eqref{eq:approx_ub} is an approximate upper bound on $m^*$ is based on the following rationale: for $k = 1$, the NPKWT reduces to a fixed-sample size test after the first randomized stopping decision. That is, given that the test continues, it will take exactly $\lfloor b_1^* \rfloor$ additional samples. Consequently, $b_1^*$ needs to be sufficiently large for the additional samples to be useful with high probability. For $k \gg 1$, in contrast, the sample size remains variable after a randomized stopping decision and can be adjusted again after the next one and so on. This leads to a more conservative increase in the remaining sample size at any particular time. Therefore, $b_1^*$ is typically larger than $m^*$. 
However, in the next section, it will be shown via a numerical counterexample that $b_1^*$, and in turn its approximation in \eqref{eq:approx_lb}, is not an upper bound on $m^*$ in general. 

For the second approximation, assume that $z \to \infty$. In this case, the second term in \eqref{eq:rho0_approx} goes to zero and the right-hand side of \eqref{eq:approx_ub} reduces to
\begin{equation*}
	\frac{1}{b}\left( g(z) - \Phi\biggl( \frac{\log z + \mu_0 b}{\sigma_0 \sqrt{b}} \biggr)\right) = \frac{1}{b} \Phi\biggl( -\frac{\log z + \mu_0 b}{\sigma_0 \sqrt{b}} \biggr),
\end{equation*}
where we used the fact that $g(z) = 1$ for $z \geq 1$. In Appendix~\ref{apx:proof_lb_approx}, it is shown that 
\begin{equation}
	\argmax_{b \geq 1} \, \frac{1}{b} \Phi\biggl( -\frac{\log z + \mu_0 b}{\sigma_0 \sqrt{b}} \biggr) \geq -\frac{\log z}{\mu_0} - \frac{\sigma_0^2}{\mu_0^2}.
	\label{EQ:BOPT_LB0}
\end{equation}
Note that the first term on the right-hand side is positive for $\log z > 0$ since $\mu_0 < 0$. The same line of arguments can be applied in the case $z \to 0$ to obtain the bound
\begin{equation}
	\argmax_{b \geq 1} \, \frac{1}{b} \Phi\biggl( \frac{\log z + \mu_1 b}{\sigma_1 \sqrt{b}} \biggr) \geq -\frac{\log z}{\mu_1} - \frac{\sigma_1^2}{\mu_1^2},
	\label{eq:bopt_lb1}
\end{equation}
where the first term is positive for $\log z < 0$ since $\mu_1 > 0$. These bounds, together with the trivial bound $m^* \geq 1$, yield our second approximation.
\begin{approximation}[Approximate Lower Bound]
	\begin{equation}
		m^*(z) \gtrsim \max\biggl\{ - \frac{\log z}{\mu_1} - \frac{\sigma_1^2}{\mu_1^2} \,,\, 1 \,,\, -\frac{\log z}{\mu_0} - \frac{\sigma_0^2}{\mu_0^2}\biggr\}.
		\label{eq:approx_lb}
	\end{equation}
	\label{apr:lower_bound}
\end{approximation}
The claim that the right-hand side of \eqref{eq:approx_lb} is an approximate lower bound is based on two observations: first, by construction, the approximation lower bounds $b_1^*$ for $z \to \infty$ and $z \to 0$, respectively. Second, it follows from Wald's identity \cite[Eq.~(3:55)]{Wald1947} that it also provides a lower bound on the expected number of samples required for an SPRT to reverse a wrong preference, that is, to either lower the log-likelihood ratio from $\log z > 0$ to $0$ under $\Hcal_0$, or to increase it from $\log z < 0$ to $0$ under $\Hcal_1$. More specifically, the $\log z$-terms in \eqref{EQ:BOPT_LB0} and \eqref{eq:bopt_lb1} correspond to the expected sample sizes when the log-likelihood ratio is assumed to hit $0$ exactly, and the constant terms are second-order corrections. Interestingly, in sequential detection the correction terms are typically positive and correct for the expected \emph{overshoot}. Here, the correction terms are negative, suggesting that $m^*(z)$ can be smaller than Wald's uncorrected approximation. The numerical examples in the next section show that this is indeed the case.

Finally, we would like to remark that the presented approximations can be improved in various ways. For example, depending on the available compute, one can tighten the approximate upper bound by using \eqref{eq:rho0_approx} as a basis to recursively obtain approximations of $\rho_k$, $k \geq 1$, and in turn $b_k^*$. Alternatively, one can obtain a good approximation of $m^*$ by dropping or modifying the correction terms in \eqref{eq:approx_lb}; see, for example, \cite{Siegmund1985} for various higher-order corrections. In general, we believe that interpreting $m^*(z)$ as the expected number of samples required for an SPRT to return from an excursion to the wrong side of the decision threshold is a good starting point for the derivation of approximate stopping policies for the NPKWT.

\section{Numerical Examples} 
\label{sec:examples}				 

In this section, we present and discuss two concrete numerical examples of the NPKWT. In the first example, we test for a shift in the success probability of a Bernoulli distribution:
\begin{equation}
	\begin{aligned}
		\Hcal_0 &\colon \bm{X} \text{ i.i.d.\ according to } \text{Bernoulli}(\sfrac{1}{3}), \\
		\Hcal_1 &\colon \bm{X} \text{ i.i.d.\ according to } \text{Bernoulli}(\sfrac{2}{3}). \\
	\end{aligned}
	\label{eq:bernoulli_example}
\end{equation}
In the second example, we test for a shift in the mean of a normal distribution:
\begin{equation}
	\begin{aligned}
		\Hcal_0 &\colon \bm{X} \text{ i.i.d.\ according to } \mathcal{N}(0, 1), \\
		\Hcal_1 &\colon \bm{X} \text{ i.i.d.\ according to } \mathcal{N}(1, 1). \\
	\end{aligned}
	\label{eq:normal_example}
\end{equation}
Both examples are simple, but well-suited to illustrate the NPKWT. 

For the unlimited randomness case, $k \to \infty$, the results presented in this section were obtained by numerically solving \eqref{eq:rho_bellman_discrete_llr} and in turn \eqref{eq:m_opt}. For the two examples above, \eqref{eq:rho_bellman_discrete_llr} simplifies to
\begin{equation}
	\frac{g(2^\lambda) - \varrho(\lambda, n)}{n} = \sup_{m \geq n} \frac{g(2^\lambda) - \frac{2}{3} \varrho\bigl(\lambda - 1 , m - 1\bigr) - \frac{1}{3} \varrho\bigl(\lambda + 1, m - 1\bigr)}{m}
	\label{eq:rho_bernoulli}
\end{equation}
and
\begin{equation}
	\frac{g(e^\lambda) - \varrho(\lambda, n)}{n} = \sup_{m \geq n} \frac{g(e^\lambda) - \Ebb_{\mathcal{N}(0, 1)}\bigl[\varrho\bigl(\lambda + X - \frac{1}{2}, m - 1\bigr)\bigr]}{m},
	\label{eq:rho_gauss}
\end{equation}
respectively. Note that we used the binary logarithm in \eqref{eq:rho_bernoulli} and the natural logarithm in \eqref{eq:rho_gauss}. 

All results presented in this section were obtained by iteratively calculating $\rho_{k+1}$ from $\rho_k$. The basis for the iteration was $\rho_0$, which can be calculated explicitly for both examples. Both $\log z$ and $c$ were discretized on a regular grid. For a given value of $c$, $\varrho_k(\bullet, c)$ was evaluated on the given log-likelihood grid, and then approximated using cubic splines. The maximization over $b_k$ was performed by an exhaustive search over the $c$-grid. For the unlimited-randomness case, the iteration was considered converged when $\lVert \rho_k - \rho_{k-1} \rVert_\infty < 10^{-6}$ on the given grid, where $\lVert \bullet \rVert_\infty$ denotes the maximum norm. Given its stronger properties, the iteration steps can be slightly simplified if one is only interested in the limit $\rho$. We will not go into more details of the numerical solution here. Python code to reproduce the results and figures in this section has been made publicly available \cite{github}. 

\subsection{Optimal Cost} %

\begin{figure}
	\centering
	\includegraphics{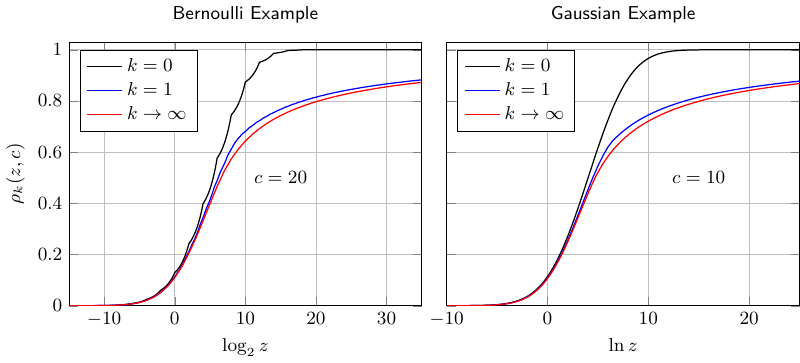}
	\caption{Optimal cost $\rho_k(z, c)$ as a function of $z$ for fixed $c$}
	\label{fig:rho_z}
\end{figure}

First, we focus on the optimal cost function, $\rho$, and its limited randomness counterpart, $\rho_k$. In Fig.~\ref{fig:rho_z}, the costs of the NPKWTs for the hypotheses in \eqref{eq:bernoulli_example} and \eqref{eq:normal_example} are plotted as functions of $z$ for a fixed value of $c$. Here, we used $c = 20$ for the Bernoulli example and $c = 10$ for the Gaussian example. By inspection, the biggest reduction in cost happens when going from $k = 0$ to $k = 1$. In light of the discussion in Section~\ref{sec:discussion}, this indicates that a single randomization is sufficient to overcome the bottleneck of the test statistic ``getting stuck'' on the wrong side of the decision threshold. Allowing for an unlimited number of randomization uses, $k \to \infty$, further reduces the cost, but the effects are less significant.

\begin{figure}
	\centering
	\includegraphics{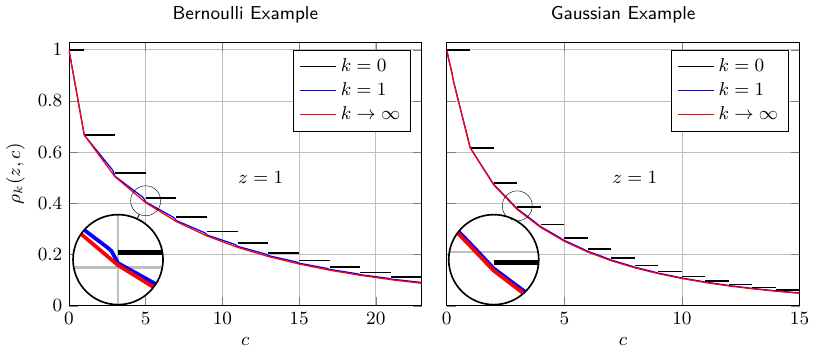}
	\caption{Optimal cost $\rho_k(z, c)$ as a function of $c$ for fixed $z$}
	\label{fig:rho_c}
\end{figure}

In Fig.~\ref{fig:rho_c}, the optimal costs of the same NPKWTs are plotted as functions of $c$ for a fixed value of $z$. In both examples, we set $z = 1$, so that the cost is the sum error probability. For $k = 0$, the optimal test is an FSST, and $\rho_0(z, \bullet)$ is a step function with discontinuities contained in the positive integers---compare Lemma~\ref{lm:rho_0}. Note that in case of the Bernoulli example, increasing the sample size from an odd number to the next even number can affect the individual error probabilities, depending on the randomization of the decision rule, but cannot reduce their sum. Accordingly, the discontinuities of the cost function are located at the odd integers in the plot on the left-hand side of Fig.~\ref{fig:rho_c}. For $k = 1$, the cost function becomes smooth, but still admits small bumps at the discontinuities of $\rho_0(z, \bullet)$, making it non-convex and nonlinear on the intervals between the discontinuities. For $k \to \infty$, the bumps are ironed-out and the cost function becomes convex and piecewise linear---compare Lemma~\ref{lm:rho_properties}. By inspection, the reduction in cost is again most significant when going from $k = 0$ to $k = 1$. However, it takes some iterations for $\rho_k(z, \bullet)$ to become convex and piecewise linear, which are critical properties for Theorem~\ref{th:optimal_stopping} to hold. Finally, note that the reduction in cost when increasing $k$ is larger for larger values of $z$---compare Fig.~\ref{fig:rho_z}. We chose $z = 1$ since it allows for a clear interpretation of the cost function.

\subsection{Optimal Expected Remaining Sample Size} %

\begin{figure}
	\centering
	\includegraphics{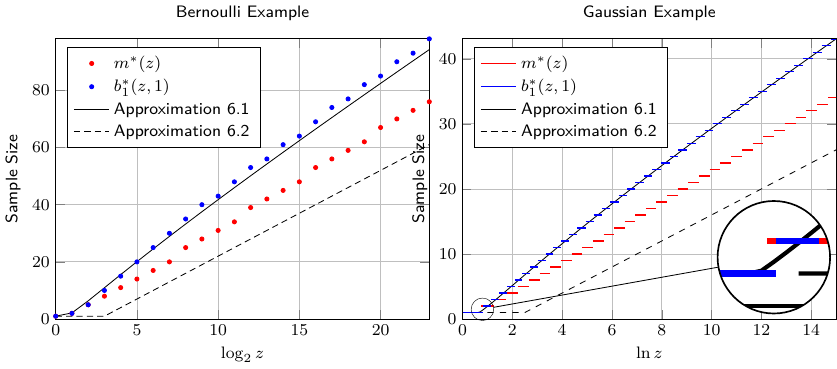}
	\caption{Optimal expected remaining sample size $m^*$ as a function of $z$}
	\label{fig:m_star}
\end{figure}

In Fig.~\ref{fig:m_star}, the function $m^*$ in \eqref{eq:m_opt} is plotted. Since the log-likelihood increment distributions are symmetric in both examples, $m^*$ is symmetric about zero. Therefore, we only plot $m^*$ for $\log z \geq 0$. Moreover, in the Bernoulli example, the log-likelihood ratio only takes integer values, so we only plot $m^*$ at these.

By inspection, $m^*$ is approximately linear in $\log z$ in both cases. This is expected in light of the discussion in the previous sections. For comparison, we also show the optimal single-randomization sample size, $b_1^*(z, 1)$, and the two approximations proposed in Section~\ref{sec:approximations}. Since the approximation in \eqref{eq:rho0_approx} is exact in the Gaussian example, the approximate upper bound coincides with $b_1^*(z, 1)$ in the plot on the right. In the Bernoulli example, the approximation is no longer exact but slightly differs from $b_1^*(z, 1)$. In general, the approximations are useful in the two examples considered here. Especially for larger $\log z$, the optimal $m^*$ is located well within the cone spanned by the two approximate bounds. In fact, the mid-point between the two bounds provides a good approximation of $m^*$ for a large range of log-likelihood ratio values. 

Interestingly, it can be seen in the magnified inset on the right that there is a small interval around $\ln z = 1.5$ on which $m^*$ exceeds both $b_1^*(z, 1)$ and the corresponding approximation. However, at this point, we cannot say with certainty if this is really a counterexample or merely a numerical artifact. In general, the question of whether $b_k^*(z, c)$ is monotonic in $k$ arguably warrants closer inspection.

\subsection{Properties of the NPKWT} %

\begin{figure}
	\centering
	\includegraphics{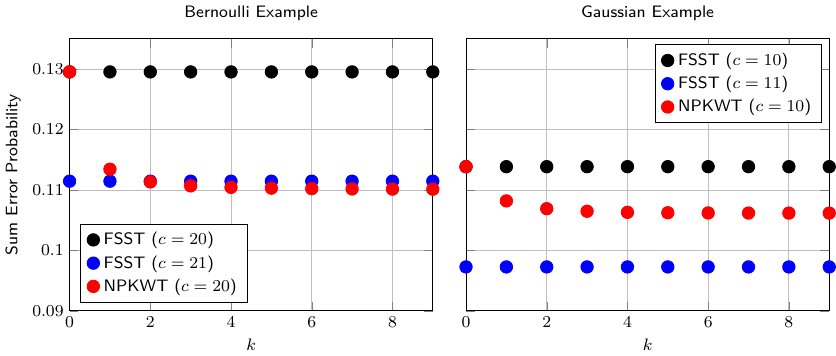}
	\caption{Sum error probability as a function of $k$}
	\label{fig:p_err_sum}
\end{figure}

In this section, we study selected properties of the NPKWT in more detail. First, the sum error probability of an NPKWT for the hypotheses in \eqref{eq:bernoulli_example} and  \eqref{eq:normal_example} is plotted as a function of the number of allowed randomization uses in Fig.~\ref{fig:p_err_sum}. Again, we used $c = 20$ for the Bernoulli example and $c = 10$ for the Gaussian example. For the Bernoulli example, it can be seen that the sum error probability decreases from approximately $0.13$ for $k = 0$ to approximately $0.11$ for $k = 10$. This is a non-negligible reduction of approximately \SI{15.4}{\percent} and highlights the potential usefulness of the NPKWT in practice. In the Gaussian example, the sum error probability decreases from approximately $0.114$ for $k = 0$ to approximately $0.106$ for $k = 10$. This is a reduction of approximately \SI{12.3}{\percent}, which  is less significant, but still noteworthy. 

For reference, the sum error probabilities of two FSSTs are plotted for sample sizes $c$ and $c + 1$. For $k = 0$, the NPKWT reduces to the FSST and its error probabilities match. Interestingly, while the FSST with a sample size of $c + 1$ achieves a lower sum error probability than the NPKWT in the Gaussian example, this is not the case in the Bernoulli example. This disproves a conjecture we had at an early stage of this work, namely, that $\rho_0(z, c) \leq \rho(z,c) \leq \rho_0(z, c+1)$. Clearly, the second bound does not hold in general. However, some insights could be gained from investigating conditions under which it does.  

\begin{figure}
	\centering
	\includegraphics{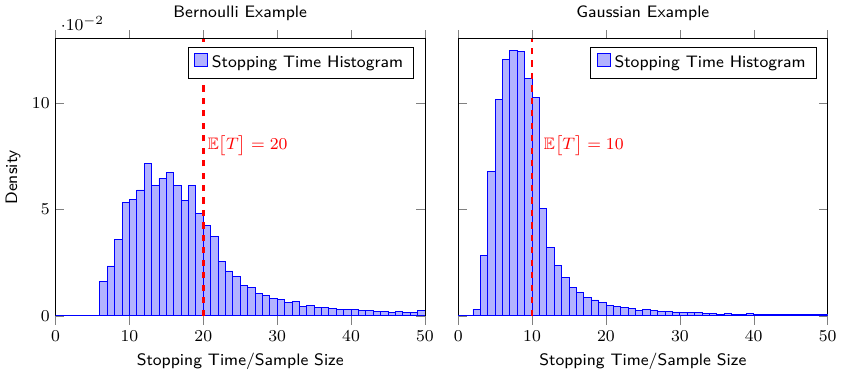}
	\caption{Histograms of the stopping times of the NPKWTs under the respective null hypotheses}
	\label{fig:stopping_time}
\end{figure}

\begin{table}
	\centering
	\caption{Statistics of the stopping time histograms in Fig.~\ref{fig:stopping_time}}
	\begin{tabular}{lrrrrr} 
		\toprule
		Example & Mean & Median & Minimum & Maximum & Standard Deviation \\
		\midrule
		Bernoulli & 19.94  & 16 & 6 & 2150 & 29.42 \\
		Gaussian & 10.03 & 8 & 1 & 2966 & 17.15 \\
		\bottomrule
	\end{tabular}
	\label{tab:stopping_time_statistics}
\end{table}

Fig.\ref{fig:stopping_time} shows (right-truncated) histograms of the stopping times of the two NPKWTs with $z = 1$ under the respective null hypotheses. Again, we used $c = 20$ for the Bernoulli example and $c = 10$ for the Gaussian example. The depicted results are based on $10^5$ Monte Carlo simulations. Both stopping time distributions are approximately of the same shape, with a single mode just below the targeted expected sample size and a pronounced right tail. Interestingly, while the NPKWT in the Gaussian example can stop at any time, its counterpart in the Bernoulli example has a minimum sample size of six. This is a consequence of the bounded log-likelihood ratio increments under the Bernoulli hypotheses. It takes at least six samples to reach a log-likelihood ratio value that is large in comparison to the remaining number of samples. 

Selected summary statistics of the two stopping time histograms in Fig.~\ref{fig:stopping_time} are given in Table~\ref{tab:stopping_time_statistics}. Arguably, the one statistic that stands out is the maximum stopping time observed over $10^5$ Monte Carlo runs, which is well over \num{2000} samples in both cases. On the one hand, this illustrates the fact that, despite having a bounded expected sample size, the NPKWT is untruncated in general and can, in principle, use arbitrarily many samples for any individual run. On the other hand, for sample sizes of this order the magnitude of the log-likelihood ratio becomes extremely large and numerical issues can arise when evaluating $\rho$ and $m^*$. Therefore, we urge the reader to take the exact numbers with a grain of salt.

\begin{figure}
	\centering
	\includegraphics{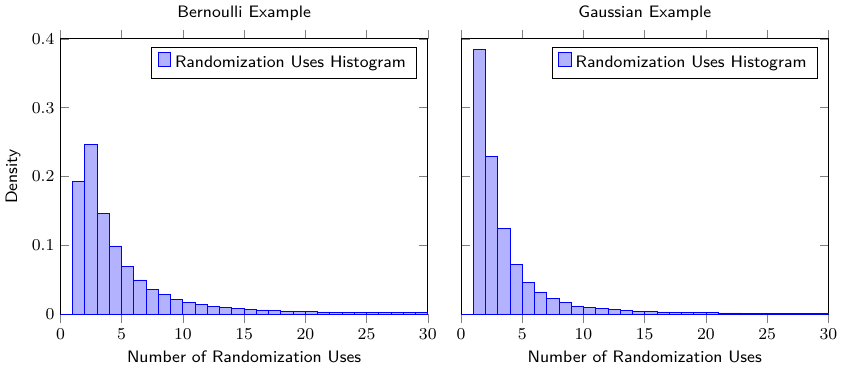}
	\caption{Histograms of the number of randomization uses of the NPKWTs under the respective null hypotheses}
	\label{fig:randomization}
\end{figure}

\begin{table}
	\centering
	\caption{Statistics of the randomization use histograms in Fig.~\ref{fig:stopping_time}}
	\begin{tabular}{lrrrrr} 
		\toprule
		Example & Mean & Median & Minimum & Maximum & Standard Deviation \\
		\midrule
		Bernoulli & 6.18 & 3 & 1 & 1201 & 17.17 \\
		Gaussian & 4.01 & 2 & 1 & 2168 & 12.40 \\
		\bottomrule
	\end{tabular}
	\label{tab:randomization_statistics}
\end{table}

Histograms of the number of randomization uses by the NPKWT are shown in Fig.~\ref{fig:randomization}. Selected summary statistics of these histograms are given in Table~\ref{tab:randomization_statistics}. 
By inspection, most instances of the NPKWT stop after at most three randomization uses in the Bernoulli case and two randomization uses in the Gaussian case. This observation confirms that in most instances a small number of randomized stopping decisions are sufficient to implement the NPKWT. However, the histograms also show large variances and significant right tails, again highlighting the untruncated nature of the NPKWT. As before, the maxima given in Table~\ref{tab:randomization_statistics}, especially in the Gaussian case, should be interpreted with caution as they could be affected by numerical noise. 

We conclude this section with an illustration of the stopping policy of the NPKWT for two specific sequences of observations. First, consider the NPKWT for the hypotheses in \eqref{eq:normal_example}, and assume that $z = 1$ and $c = 20$. If $x_t = \sfrac{1}{2}$ for all $t \geq 1$, the observations do not provide any information about the true hypothesis, and the likelihood ratio remains constant, namely, $z_t = z = 1$ for all $t \geq 1$. In turn, $m^*(z_t) = m^*(1) = 1$ for all $t \geq 1$. Using the update rules in Theorem~\ref{th:optimal_stopping}, the test simply decrements $c_t$ after each sample and stops when $c_t = 0$. That is, the NPKWT reduces to an FSST with $c = 20$ samples. This is an example of how the NPKWT avoids large stopping times in scenarios where a conventional SPRT breaks down because the observations do not provide sufficient evidence to cross a decision threshold. 

\begin{table}
	\centering
	\caption{Continuation probabilities of the NPKWT for the hypotheses in \eqref{eq:bernoulli_example} when observing the all-ones sequence.}
	\begin{tabular}{cccccccccccc}
		\toprule
		$t$ & $\leq$ 2 & 3 & 4 & 5 & 6 & 7 & 8 & 9 & 10 & 11 & 12 \\[2pt]
		$1 - \psi(2^t)$ & $1$ & $\sfrac{7}{8}$ & $\sfrac{7}{11}$ & $\sfrac{10}{14}$ & $\sfrac{13}{17}$ & $\sfrac{16}{20}$ & $\sfrac{19}{25}$ & $\sfrac{24}{28}$ & $\sfrac{27}{31}$ & $\sfrac{30}{34}$ & $\sfrac{33}{39}$ \\
		\bottomrule
	\end{tabular}
	\label{tab:continuation_probabilities}
\end{table}

Next, consider the NPKWT for the hypotheses in \eqref{eq:bernoulli_example}, and assume that $z = 1$ and $c = 10$. If $x_t = 1$ for all $t \geq 1$, the observations all point toward the alternative hypothesis, and the likelihood ratio is monotonically increasing, namely, $z_t = 2^t$. The resulting sequence of continuation probabilities is shown in Table~\ref{tab:continuation_probabilities}. In this example, randomization starts after three samples, and the structure of the randomization rule in \eqref{eq:stopping_rule} is reflected in these probabilities: for $t \geq 3$, the continuation probabilities are of the form $\tfrac{m^*(2^{t-1}) - 1}{m^*(2^t)}$. Interestingly, the continuation probability approaches $1$ as the evidence for $\Hcal_1$ increases. This might seem counterintuitive, but is in line with the discussion in the previous section: since $m^*$ is approximately linear in $\log z$, the \emph{absolute} increment in expected remaining sample size after each likelihood ratio update, $c_t - c_{t-1}$, is approximately constant. However, \emph{relative} to $c_t$, this increment, and in turn the stopping probability, tends to zero. This illustrates how the NPKWT can delay stopping in cases where the observations are highly informative, and helps to understand the extremely large sample sizes observed in our experiments.

\section{Conclusion and Outlook} 
\label{sec:conclusion}					 

In this paper, we have formulated and solved the nonparametric Kiefer--Weiss problem. The corresponding sequential test has been shown to rely heavily on randomized stopping rules, and the optimal randomization to be defined by a function that maps the current likelihood ratio to a minimal number of expected remaining samples. The properties of the optimal test have been further explored and illustrated with numerical examples.

Questions left open in the paper include:
\begin{itemize}
	\item Can the gain of the NPKWT over the FSST be bounded? For example, is there a constant $\Delta$, possibly dependent on $P_0$ and $P_1$, such that $\rho(z,c) \leq \rho_0(z, c + \Delta)$?
	\item What can be said about the stopping time distribution of the NPKWT under a given sample distribution $\Pb$?
	\item Is the optimal sample size $b_k^*(z, c)$ in \eqref{eq:b_opt_k} monotonic in $k$?
\end{itemize}

Another possible research direction is to investigate the properties of the NPKWT for vanishing error probabilities, or, equivalently, for $c \to \infty$. Although this topic has not been addressed in this paper, we conjecture that in this asymptotic regime the optimal stopping rules simplify significantly, possibly to a variation of Approximation~\ref{apr:lower_bound}. However, we expect a formal proof to be non-trivial, as it needs to take the dynamics of $c_t$ into account. Even for large initial values $c_0$, the NPKWT can, in principle, re-enter a non-asymptotic regime as the test progresses.

Finally, we would like to emphasize that the NPKWT is not only of theoretical interest, but can be an attractive choice in practice as well. In contrast to conventional KWTs, which often have time-dependent or two-dimensional thresholds and require substantial computation, the NPKWT is essentially an augmented SPRT. Moreover, as discussed before, it is robust against deviations from the optimal mapping $m^*$, so that, in practice, approximate stopping policies such as the ones proposed in Section~\ref{sec:approximations} can be used with minimal computational overhead. Therefore, we encourage interested researchers and practitioners to identify and investigate potential applications.

\appendix

\section{Proof of Theorem~\ref{th:equivalence}} 
\label{apx:equivalence}                         

Assume that $\psi$ satisfies the condition in \eqref{eq:stopping_policy_pointwise}. It then holds that
\begin{equation*}
	\Ebb_{\Pb, \psi}\bigl[ T \bigr] = \Ebb_{\Pb}\bigl[ \Ebb_\psi[ T \,|\, \bm{X}] \bigr] = \Ebb_{\Pb}\bigl[ \tau(\bm{X}) \bigr] \leq \Ebb_{\Pb}[\, c \,] = c,
\end{equation*}
so that $\psi$ also satisfies the condition in \eqref{eq:stopping_policy_prob}. Now assume that $\psi$ violates the condition in \eqref{eq:stopping_policy_pointwise}. In this case, there exists at least one sequence $\tilde{\bm{x}}$ for which
\begin{equation*}
	\Ebb_\psi\bigl[ \tau(\tilde{\bm{x}}) \bigr] > c.
\end{equation*}
This implies that
\begin{equation*}
	\Ebb_{\bm{1}_{\tilde{\bm{x}}}, \psi}\bigl[\, T \,\bigr] = \Ebb_{\psi}\bigl[\, T \,|\, \bm{X} = \tilde{\bm{x}} \,\bigr] = \Ebb_{\psi}\bigl[\, \tau(\tilde{\bm{x}}) \,\bigr] > c,
\end{equation*}
where $\bm{1}_{\bm{x}} \in \Mcal(\Xcal^\infty)$ denotes a unit point mass on $\bm{x}$. Therefore, $\psi$ also violates the condition in \eqref{eq:stopping_policy_prob}. This completes the proof.

\section{Proof of Theorem~\ref{th:reduction}} 
\label{apx:reduction}                         

By conditioning on the stopping time we obtain
\begin{align}
	\alpha(\psi, \delta) + & z \, \beta(\psi, \delta) \notag \\
	&= \Ebb_{0, \psi, \delta}\bigl[ D \bigr] + z \, \Ebb_{1, \psi, \delta}\bigl[ 1 - D \bigr] \notag \\
	&= \sum_{t=0}^\infty \Bigl( \Ebb_{0, \delta}\bigl[ D \,|\, T = t \bigr] \Pbb_{0, \psi}\bigl[ T = t \bigr] + z \, \Ebb_{1, \delta}\bigl[ 1 - D \,|\, T = t \bigr] \Pbb_{1, \psi}\bigl[ T = t \bigr] \Bigr). \label{eq:err_t}
\end{align}
Additionally conditioning on the observed sequence yields
\begin{align}
	\Ebb_{0, \delta}\bigl[ D \,|\, T = t \bigr] & \Pbb_{0, \psi}\bigl[ T = t \bigr] \notag \\
	&= \int_{\Xcal^t} \Ebb_{0, \delta}\bigl[ D \,|\, T = t, \bm{X}_t = \bm{x}_t \bigr] \Pbb_{0, \psi}\bigl[ T = t \,|\, \bm{X}_t = \bm{x}_t \bigr] \prod_{n=1}^t p_0(x_n) \, \mathrm{d}\bm{x}_t \notag \\
	&= \int_{\Xcal^t} \delta_t(\bm{x}_t) \phi_t(\bm{x}_t) \prod_{n=1}^t p_0(x_n) \, \mathrm{d}\bm{x}_t \label{eq:err_t_0}
\end{align}
and
\begin{align}
	\Ebb_{1,\delta}\bigl[ 1 &- D \,|\, T = t \bigr] \Pbb_{1, \psi}\bigl[ T = t \bigr] \notag \\
	&= \int_{\Xcal^t} \Ebb_{1,\delta}\bigl[ 1 - D \,|\, T = t, \bm{X}_t = \bm{x}_t \bigr] \Pbb_{1,\psi}\bigl[ T = t \,|\, \bm{X}_t = \bm{x}_t \bigr] \prod_{n=1}^t p_1(x_n) \, \mathrm{d}\bm{x}_t \notag \\
	&= \int_{\Xcal^t} (1 - \delta_t(\bm{x}_t)) \phi_t(\bm{x}_t) \prod_{n=1}^t p_1(x_n) \, \mathrm{d}\bm{x}_t, \label{eq:err_t_1}
\end{align}
with $\phi_t$ and $\delta_t$ defined in \eqref{eq:phi_t} and \eqref{eq:delta_t}, respectively. Substituting \eqref{eq:err_t_0} and \eqref{eq:err_t_1} back into \eqref{eq:err_t} yields the following lower bound:
\begin{align*}
	\alpha(\psi, \delta) + &z \, \beta(\psi, \delta) \\
	&= \sum_{t=0}^\infty \int_{\Xcal^t} \delta_t(\bm{x}_t) \phi_t(\bm{x}_t) \prod_{n=1}^t p_0(x_n) + z (1 - \delta_t(\bm{x}_t)) \phi_t(\bm{x}_t) \prod_{n=1}^t p_1(x_n) \, \mathrm{d}\bm{x}_t \\
	&\geq \sum_{t=0}^\infty \int_{\Xcal^t} \min\biggl\{ \prod_{n=1}^t p_0(x_n) \,,\, z \prod_{n=1}^t p_1(x_n) \biggr\} \phi_t(\bm{x}_t) \, \mathrm{d}\bm{x}_t \\
	&= \sum_{t=0}^\infty \int_{\Xcal^t} \min\biggl\{ 1 \,,\, z \prod_{n=1}^t \frac{p_1(x_n)}{p_0(x_n)} \biggr\} \phi_t(\bm{x}_t) \, \prod_{s=1}^t p_0(x_s) \mathrm{d}\bm{x}_t \\
	&= \sum_{t=0}^\infty \Ebb_{0, \psi}\biggl[ g\biggl(z \prod_{n=1}^t \frac{p_1(X_n)}{p_0(X_n)} \biggr) \phi_t(\bm{X}_t) \biggr] \\
	&= \Ebb_{0, \psi}\Biggl[ \sum_{t=0}^\infty  g\biggl(z \prod_{n=1}^t \frac{p_1(X_n)}{p_0(X_n)} \biggr) \phi_t(\bm{X}_t) \Biggr] \\
	&= \Ebb_{0, \psi}\Biggl[ g\biggl(z \prod_{t=1}^T \frac{p_1(X_t)}{p_0(X_t)} \biggr) \Biggr],
\end{align*}
where in the second-to-last step summation and expectation can be interchanged since all summands are non-negative \cite[2.37~Fubini-Tonelli Theorem]{Folland1999}. The lower bound is attained if and only if $\delta_t$ is of the form \eqref{eq:delta_opt} in Theorem~\ref{th:reduction}. This completes the proof.

\section{Proof of Lemma~\ref{lm:existence}} 
\label{apx:existence}												
The problem in \eqref{eq:nonpa_kiwei_stopping} can be reparametrized in terms of the conditional distribution of the stopping time:
\begin{equation}
	\inf_{\phi \in \mathcal{S}} \, \Ebb_0\Biggl[\sum_{t=0}^\infty  Y_t(\bm{X}_t) \phi_t(\bm{X}_t) \Biggr] \quad \text{s.t.} \quad \sum_{t=0}^\infty t \, \phi_t(\bm{x}_t) \leq c \quad \forall \bm{x} \in \Xcal^\infty,
	\label{eq:constrained_optimal_stopping}
\end{equation}
where $\mathcal{S}$ denotes the space of stopping time distributions and
\begin{equation*}
	Y_t(\bm{X}_t) \coloneq g\biggl( z \prod_{n=1}^t \frac{p_1(X_n)}{p_0(X_n)} \biggr).
\end{equation*}
In a nutshell, the existence of a minimizer follows from the fact that \eqref{eq:constrained_optimal_stopping} is a well-defined optimal stopping problem: The cost, $Y_t$, is non-negative and bounded for all $t \geq 0$, and both the objective and the constraints are linear in $\phi$. 

More formally, we need to show that the set of feasible stopping time distributions is compact and that the objective function is lower-semiconscious. Existence of a minimizer then follows from the extreme value theorem \cite[Theorem~2.43]{Aliprantis2007}.

The seminal work on compactness of stopping times is due to Baxter and Chacon who identified the appropriate topology \cite{Baxter1977}. All statements made in what follows hold in the Baxter-Chacon (BC) topology. 

A formal proof proceeds as follows:
\begin{enumerate}
	\item The space of discrete-time stopping time distributions, $\mathcal{S}$, is sequentially compact \cite[Theorem~1.1]{Edgar1982}.
	\item Define the sequence of functions $(f_n)_{n \geq 0}$ as
	\begin{equation*}
		f_n(\phi) \coloneq \sum_{t=0}^n t \, \phi_t(\bm{x}_t).
	\end{equation*}
	For all $n \geq 0$, the function $f_n$ is linear and bounded and, therefore \cite[Chapter~VI, Theorem~1.1]{Conway1994}, continuous in $\phi$. Moreover, since $t \, \phi_t \geq 0$ for all $t \geq 0$, the sequence $(f_n)_{n \geq 0}$ is non-decreasing so that
	\begin{equation*}
		\sum_{t=0}^\infty t \, \phi_t(\bm{x}_t) = \sup_{n \geq 0} f_n(\phi).
	\end{equation*}
	Since the supremum of continuous functions is lower semicontinuous \cite[Lemma 2.41]{Aliprantis2007}, this implies that the constraint function in \eqref{eq:constrained_optimal_stopping} is lower semicontinuous.
	\item By \cite[Definition~2.8]{Rudin1987}, sublevel sets of lower semicontinuous functions are closed. Moreover, arbitrary intersections of closed sets are closed \cite[Theorem~6.1]{Munkres1974}, and closed subsets of compact sets are compact \cite[Theorem 2.35]{Rudin1976}. It follows that the subset of $\mathcal{S}$ defined by the constraints in \eqref{eq:constrained_optimal_stopping} is compact. 
	\item Since $Y_t \geq 0$ for all $t \geq 0$, it follows from Fatou's lemma \cite[Theorem 11.31]{Rudin1976} that the objective function in \eqref{eq:constrained_optimal_stopping} is lower semicontinuous.
\end{enumerate}
This concludes the proof.

\section{Proof of Lemma~\ref{lm:rho_k_properties}} 
\label{apx:rho_k_properties}                       

To show the first statement, let $\gamma \in [0, 1]$ and $z_1, z_2 \geq 0$. For all $k \geq 0$ it holds that
\begin{align*}
	\rho_k(\gamma z_1 + (1 - \gamma) z_2, c) &= \inf_{\psi \in \Ical_{c,k}^\infty} \alpha(\psi) + (\gamma z_1 + (1 - \gamma) z_2) \beta(\psi) \\
	&= \inf_{\psi \in \Ical_{c,k}^\infty} \gamma \bigl( \alpha(\psi) + z_1 \beta(\psi) \bigr) + (1 - \gamma) \bigl( \alpha(\psi) + z_2 \beta(\psi) \bigr) \\
	&\geq \gamma \inf_{\psi \in \Ical_{c,k}^\infty} \alpha(\psi) + z_1 \beta(\psi) + (1 - \gamma) \inf_{\psi \in \Ical_{c,k}^\infty} \alpha(\psi) + z_2 \beta(\psi) \\
	&= \gamma \rho_k(z_1, c) + (1 - \gamma) \rho_k(z_2, c).
\end{align*} 
This proves that $\rho_k$ is concave in $z$. Now, assume that $z_2 > z_1$. Since $\alpha$ and $\beta$ are non-negative
\begin{equation*}
	\alpha(\psi) + z_2 \, \beta(\psi) \geq \alpha(\psi) + z_1 \, \beta(\psi)
\end{equation*}
for any stopping policy $\psi$. Taking the infimum over $\Ical_{c,k}^\infty$ on both sides yields
\begin{align*}
	\inf_{\psi \in \Ical_{c,k}^\infty} \alpha(\psi) + z_2 \, \beta(\psi) &\geq \inf_{\psi \in \Ical_{c,k}^\infty} \alpha(\psi) + z_2 \, \beta(\psi) \\
	\rho_k(z_2, c) &\geq \rho_k(z_1, c).
\end{align*}

Finally, since $\rho_k(z, c)$ is non-increasing in $c$, it holds that
\begin{equation*}
	\rho_k(z, c) \leq \rho_k(z, 0) = \Ebb_0\biggl[ g\biggl(z \prod_{n=1}^0 \frac{p_1(X_n)}{p_0(X_n)} \biggr) \biggr] = g(z).
\end{equation*}
This completes the proof.

\section{Proof of Lemma~\ref{lm:rho_k_recursive}} 
\label{apx:rho_k_recursive}												

Consider the problem of choosing the initial stopping probability, $\psi_0$, of a test with stopping policy $\psi \in \Ical_{c,k}^\infty$. Conceptually speaking, we can distinguish between three cases:
\begin{enumerate}
	\item Stop with certainty and make a decision ($\psi_0 = 1$).
	\item Continue with certainty and save all $k$ randomization uses for later ($\psi_0 = 0$).
	\item Use a randomized stopping rule and either stop, or continue with a policy that allows for at most $k-1$ randomization uses ($0 < \psi_0 < 1$).
\end{enumerate}
The optimal choice is then the one that minimizes the expected cost. In order to formalize this argument, we define
\begin{equation*}
	\rho_k(z, c \,|\, \psi_0 = \eta) \coloneq \inf_{(\eta, \psi_{1:\infty}) \in \Ical_{c,k}^\infty} \; \Ebb_{0, \psi}\Biggl[ g\biggl( z \prod_{n=1}^T \frac{p_1(X_n)}{p_0(X_n)} \biggr) \,\bigg|\, \psi_0 = \eta \Biggr],
\end{equation*}
where $\psi_{1:\infty}$ is shorthand for the sequence $(\psi_1, \psi_2, \ldots)$. First, assume that $\psi_0 = 1$, that is, the test stops with certainty before taking the first sample. In this case, $\phi_0 = 1$ and
\begin{equation}
	\rho_k(z, c \,|\, \psi_0 = 1) = \Ebb_0\Biggl[ g\biggl( z \prod_{n=1}^0 \frac{p_1(X_n)}{p_0(X_n)} \biggr) \Biggr] = g(z).
	\label{eq:rho_k_1}
\end{equation}
Next, assume that $\psi_0 = 0$, that is, the test takes the first sample with certainty. Note that this choice is only feasible if $c \geq 1$. In this case
\begin{align}
	\rho_k(z, c \,|\, \psi_0 = 0) &= \inf_{(0, \psi_{1:\infty}) \in \Ical_{c,k}^\infty} \; \Ebb_{0, \psi}\Biggl[ g\biggl(z \prod_{n=1}^T \frac{p_1(X_n)}{p_0(X_n)} \biggr) \,\bigg|\, \psi_0 = 0 \Biggr] \notag \\
	&= \inf_{(0, \psi_{1:\infty}) \in \Ical_{c,k}^\infty} \;  \Ebb_0\Biggl[ \Ebb_{0, \psi}\Biggl[ \sum_{t=1}^\infty g\biggl(z \prod_{n=1}^T \frac{p_1(X_n)}{p_0(X_n)} \biggr) \,\bigg|\, \psi_0 = 0, X_1 \Biggr] \Biggr] \notag \\
	&= \Ebb_0\Biggl[ \inf_{(0, \psi_{1:\infty}) \in \Ical_{c,k}^\infty} \; \Ebb_{0, \psi}\Biggl[ \sum_{t=1}^\infty  g\biggl( z\frac{p_1(X_1)}{p_0(X_1)} \prod_{n=2}^T \frac{p_1(X_n)}{p_0(X_n)} \biggr) \,\bigg|\, \psi_0 = 0, X_1 \Biggr] \Biggr],
	\label{eq:rho_k_0_intermediate}
\end{align}
where in the last step expectation and minimization can be interchanged since all $\psi_t$, $t \geq 1$ are measurable functions of $X_1$, that is, the infimum can be taken pointwise---compare \cite[Theorem~14.60]{Rockafellar2009}. Now, let $\bm{X}'$ and $\psi'$ denote shifted versions of $\bm{X}$ and $\psi$, respectively. More specifically, we define $X'_t \coloneq X_{t+1}$ and $\psi'_t \coloneq \psi_{t+1}$ for all $t \geq 0$. By construction, it holds that $\psi' \in \Ical_{c-1,k}^\infty$. This follows since $\psi \in  \Ical_{c,k}^\infty$ and
\begin{align*}
	\Ebb_{(0, \psi_{1:\infty})}[\tau(\bm{x})] &= \sum_{t=1}^\infty t \, \phi_t(\bm{x}_t) \\
	&= \sum_{t=0}^\infty (t + 1) \, \phi'_t(\bm{x}'_t) \\
	&= \sum_{t=0}^\infty t \, \phi'_t(\bm{x}'_t) + \sum_{t=0}^\infty \phi'_t(\bm{x}'_t) \\
	&= \Ebb_{\psi'}[\tau(\bm{x})] + 1,
\end{align*}
where $\phi'_t$ is shorthand for $\phi_t$ in \eqref{eq:phi_t} evaluated at $\psi = \psi'$. The inner minimization in \eqref{eq:rho_k_0_intermediate} can now be written as
\begin{align}
	\inf_{(0, \psi_{1:\infty}) \in \Ical_{c,k}^\infty} & \Ebb_{0, \psi}\Biggl[ \sum_{t=1}^\infty  g\biggl( z\frac{p_1(X_1)}{p_0(X_1)} \prod_{n=2}^T \frac{p_1(X_n)}{p_0(X_n)} \biggr) \,\Big|\, \psi_0 = 0, X_1 \Biggr] \notag \\
	&= \inf_{(0, \psi_{1:\infty}) \in \Ical_{c,k}^\infty} \Ebb_{0, \psi}\Biggl[ \sum_{t=1}^\infty g\biggl( z\frac{p_1(X_1)}{p_0(X_1)} \prod_{n=2}^{t} \frac{p_1(X_n)}{p_0(X_n)} \biggr) \phi_t(\bm{X}_t) \,\Big|\, \psi_0 = 0, X_1 \biggr] \notag \\
	&= \inf_{\psi' \in \Ical_{c-1, k}^\infty} \Ebb_{0, \psi'}\Biggl[ \sum_{t=0}^\infty  g\biggl( z\frac{p_1(X'_0)}{p_0(X'_0)} \prod_{n=1}^t \frac{p_1(X'_n)}{p_0(X'_n)} \biggr) \phi'_t(\bm{X}'_t) \,\Big|\, X'_0 \Biggr] \notag \\
	&= \inf_{\psi' \in \Ical_{c-1, k}^\infty} \Ebb_{0, \psi'}\Biggl[ g\biggl( z\frac{p_1(X'_0)}{p_0(X'_0)} \prod_{n=1}^T \frac{p_1(X'_n)}{p_0(X'_n)} \biggr) \,\Big|\, X'_0 \Biggr] \notag \\
	&= \rho_k\biggl( z \frac{p_1(X'_0)}{p_0(X'_0)}, c - 1\biggr) \notag \\
	&= \rho_k\biggl( z \frac{p_1(X_1)}{p_0(X_1)}, c - 1\biggr). \label{eq:rho_k_0_conditioned}
\end{align}
Substituting \eqref{eq:rho_k_0_conditioned} back into \eqref{eq:rho_k_0_intermediate} and making the restriction to $c \geq 1$ explicit yields
\begin{equation}
	\rho_k(z, c \,|\, \psi_0 = 0) = \begin{dcases}
		\Ebb_0\biggl[\rho_k\biggl( z \frac{p_1(X)}{p_0(X)}, c - 1\biggr)\biggr], & c \geq 1, \\[2pt]
		\text{undefined}, & c < 1,
	\end{dcases}
	\label{eq:rho_k_0}
\end{equation}
where we dropped the index of $X$ since all $X_t$ are i.i.d.

Finally, assume that $\eta \in (0,1)$, that is, the test stops with probability $\eta$ and continues with probability $1 - \eta$. In this case, the conditional cost is a mixture of the form
\begin{equation}
	\eta \, g(z) + (1 - \eta) \, \Ebb_0\biggl[\rho_{k-1}\biggl(z \frac{p_1(X)}{p_0(X)}, b - 1 \biggr)\biggr],
	\label{eq:rho_k_eta_b}
\end{equation}
with $b \geq 1$. To see this, consider the following. If the test stops, \eqref{eq:rho_k_1} holds and the resulting cost is $g(z)$. If the test continues, the arguments leading to \eqref{eq:rho_k_0} apply, but with two key differences: first, the expected sample size of the test conditioned on $S_0 = 0$ can exceed $c$, as long as the \emph{unconditional} expected sample size is bounded by $c$; second, the number of remaining randomization uses for $t \geq 1$ reduces to $k - 1$. In combination, this means that $\psi' \in \Ical_{b-1, k-1}$, where $b$ is a free variable. Substituting $\psi' \in \Ical_{b-1, k-1}$ for $\psi' \in \Ical_{c-1, k}$ in the steps leading to \eqref{eq:rho_k_0_conditioned} yields the expected cost of continuing in \eqref{eq:rho_k_eta_b}.

We now establish the feasible region for $\eta$ and $b$. First, it trivially needs to hold that $b \geq 1$ since continuing the test implies taking at least one more sample. Moreover, in order for $(\eta, \psi_{1:\infty}) \in \Ical_{c,k}^\infty$, it needs to hold that
\begin{align*}
	\Ebb_{(\eta, \psi_{1:\infty})}\bigl[\tau(\bm{x})\bigr] &\leq c \\
	\eta \, \Ebb_{\psi_{1:\infty}}\bigl[\tau(\bm{x}) \,|\, S_0 = 1 \bigr] + (1 - \eta) \, \Ebb_{\psi_{1:\infty}}\bigl[\tau(\bm{x}) \,|\, S_0 = 0 \bigr] &\leq c \\
	\eta \, 0 + (1 - \eta) \, \bigl(1 + \Ebb_{\psi'}\bigl[\tau(\bm{x}')\bigr] \bigr) &\leq c\\
	\Ebb_{\psi'}\bigl[\tau(\bm{x}')\bigr] &\leq \frac{c}{1 - \eta} - 1 
\end{align*}
for all $\bm{x} \in \Xcal^\infty$, where $\bm{x}'$ and $\psi'$ are as defined above. Since $\psi' \in \Ical_{b-1, k-1}^\infty$, this implies 
\begin{equation}
	b \leq \frac{c}{1 - \eta}. 
	\label{eq:b_of_eta}
\end{equation}
In order to determine the optimal $b$, it is useful to treat the cases $c < 1$ and $c \geq 1$ separately. First, assume that $c \geq 1$. In this case, minimizing the expression in \eqref{eq:rho_k_eta_b} over $b$ yields
\begin{align}
	\rho_k(z, c \,|\, \psi_0 = \eta) &= \inf_{1 \leq b \leq \frac{c}{1 - \eta}} \, \eta \, g(z) + (1 - \eta) \, \Ebb_0\biggl[\rho_{k-1}\biggl(z \frac{p_1(X)}{p_0(X)}, b - 1 \biggr)\biggr] \notag \\
	&= \eta \, g(z) + (1 - \eta) \, \Ebb_0\biggl[\rho_{k-1}\biggl(z \frac{p_1(X)}{p_0(X)}, \frac{c}{1 - \eta} - 1 \biggr)\biggr]
	\label{eq:rho_k_eta_min_c_large}
\end{align}
where the second equality holds since $\rho_{k-1}(z, c)$ is non-increasing in $c$---see Lemma~\ref{lm:rho_k_properties}. 

Now assume that $c < 1$. In this case, the right-hand side of \eqref{eq:b_of_eta} needs to be larger than or equal to one in order for any feasible $b$ to exist:
\begin{equation}
	\rho_k(z, c \,|\, \psi_0 = \eta) = \begin{dcases}
		\eta \, g(z) + (1 - \eta) \, \Ebb_0\biggl[\rho_{k-1}\biggl(z \frac{p_1(X)}{p_0(X)}, \frac{c}{1 - \eta} - 1 \biggr)\biggr], & \frac{c}{1-\eta} \geq 1 \\
		\text{undefined}, & \frac{c}{1-\eta} < 1 
	\end{dcases} 
	\label{eq:rho_k_eta_min_c_small}
\end{equation}
The expressions in \eqref{eq:rho_k_eta_min_c_large} and \eqref{eq:rho_k_eta_min_c_small} can be combined by parametrizing them as follows:
\begin{align}
	\rho_k\Bigl(z, c \,\Big|\, \psi_0 = 1 - \frac{c}{b}\Bigr) &= \Bigl(1 - \frac{c}{b} \Bigr) \, g(z) + \frac{c}{b} \, \Ebb_0\biggl[\rho_{k-1}\biggl(z \frac{p_1(X)}{p_0(X)}, b - 1 \biggr)\biggr] \notag \\
	&= g(z) - c \, \frac{g(z) - \Ebb_0\Bigl[\rho_{k-1}\Bigl(z \frac{p_1(X)}{p_0(X)}, b - 1 \Bigr)\Bigr]}{b},
	\label{eq:rho_k_eta}
\end{align}
where $b \geq 1$ if $c < 1$ and $b > c$ if $c \geq 1$.

The result in the lemma can now be obtained by combining \eqref{eq:rho_k_1}, \eqref{eq:rho_k_0} and \eqref{eq:rho_k_eta}. First, assume that $c < 1$. In this case, it holds that
\begin{align*}
	\rho_k(z, c) &= \inf_{\eta \in (0,1]} \rho_k(z, c \,|, \psi_0 = \eta) \\
	&= \min\Bigl\{ \, \inf_{b \geq 1} \rho_k\Bigl(z, n \,|\, \psi_0 = 1 - \frac{c}{b}\Big) \,,\, \rho_k(z, n \,|\, \psi_0 = 1) \Bigr\} \\
	&= \min\biggl\{\, \inf_{b \geq 1} g(z) - c \,  \frac{g(z) - \Ebb_0\bigl[\rho_{k-1}\bigl(z \frac{p_1(X)}{p_0(X)}, b - 1 \bigr)\bigr]}{b} \,,\, g(z) \biggr\} \\
	&= g(z) - c \, \sup_{b \geq 1}  \frac{g(z) - \Ebb_0\bigl[\rho_{k-1}\bigl(z \frac{p_1(X)}{p_0(X)}, b - 1 \bigr)\bigr]}{b},
\end{align*}
where we used the fact that $\lim_{b \to \infty} \rho_k\bigl(z, c \,\big|\, \psi_0 = 1 - \frac{c}{b}\bigr) = \rho_k(z, c \,|\, \psi_0 = 1) = g(z)$ so that the case $\psi_0 = 1$ does not need to be treated separately. For $c \geq 1$ it holds that
\begin{align}
	\rho_k(z, c) &= \inf_{\eta \in [0,1]} \rho_k(z, c \,|, \psi_0 = \eta) \notag \\
	&= \min\Bigl\{ \rho_k(z, n \,|\, \psi_0 = 0) \,,\, \inf_{b > c} \rho_k\Bigl(z, n \,|\, \psi_0 = 1 - \frac{c}{b}\Big) \,,\, \rho_k(z, n \,|\, \psi_0 = 1) \Bigr\} \notag \\
	&= \min\biggl\{ \Ebb_0\biggl[\rho_k\biggl( z \frac{p_1(X)}{p_0(X)}, c - 1\biggr)\biggr] \,,\, \notag \\
	& \hspace{0.6in} \inf_{b > c} g(z) - c \,  \frac{g(z) - \Ebb_0\bigl[\rho_{k-1}\bigl(z \frac{p_1(X)}{p_0(X)}, b - 1 \bigr)\bigr]}{b} \,,\, g(z) \biggr\} \notag \\
	&= \min\biggl\{\Ebb_0\biggl[\rho_k\biggl( z \frac{p_1(X)}{p_0(X)}, c - 1\biggr)\biggr] \,,\, \notag \\
	& \hspace{0.6in} g(z) - c \, \sup_{b > c}  \frac{g(z) - \Ebb_0\bigl[\rho_{k-1}\bigl(z \frac{p_1(X)}{p_0(X)}, b - 1 \bigr)\bigr]}{b} \biggr\},
	\label{eq:rho_k_1_inf}
\end{align}
where again the case $\psi_0 = 1$ is included via the limit $b \to \infty$. 

It remains to show that the supremum is attained at a finite maximizer, and that $c$ can be included in the maximization domain in \eqref{eq:rho_k_1_inf}. To see that a finite maximizer exists, note that
\begin{equation*}
	g(z) - \Ebb_0\biggl[\rho_{k-1}\biggl(z \frac{p_1(X)}{p_0(X)}, b - 1 \biggr)\biggr] \geq \rho_{k-1}(z, 0) - \rho_{k-1}(z, b - 1 \bigr) \geq 0 
\end{equation*}
for all $b \geq 1$. Here, the first inequality follows from Jensen's inequality together with the fact that $\rho_k(\bullet , c)$ is concave, and the second holds since $\rho_k(z , \bullet)$ is non-increasing. However,
\begin{equation*}
	\lim_{b \to \infty} \, \frac{g(z) - \Ebb_0\bigl[\rho_{k-1}\bigl(z \frac{p_1(X)}{p_0(X)}, b - 1 \bigr)\bigr]}{b} \leq \lim_{b \to \infty} \, \frac{g(z)}{b} = 0.
\end{equation*}
To see that $c$ can be included in the maximization domain in \eqref{eq:rho_k_1_inf}, note that
\begin{align*}
	g(z) - c \, \frac{g(z) - \Ebb_0\bigl[\rho_{k-1}\bigl(z \frac{p_1(X)}{p_0(X)}, b - 1 \bigr)\bigr]}{b} \bigg\rvert_{b = c} &= \Ebb_0\biggl[\rho_{k-1}\biggl(z \frac{p_1(X)}{p_0(X)}, c - 1 \biggr)\biggr] \\
	&\geq \Ebb_0\biggl[\rho_k\biggl(z \frac{p_1(X)}{p_0(X)}, c - 1 \biggr)\biggr],
\end{align*}
where the last inequality holds since $\rho_k$ is non-increasing in $k$. This completes the proof.

\section{Proof of Lemma~\ref{lm:integral_equation}} 
\label{apx:integral_equation}												

Taking the limit $k \to \infty$ on both sides of \eqref{eq:rho_k_recursive} yields
\begin{equation*}
	\medmuskip=-1mu
	\thickmuskip=0mu
	\rho(z, c) = \begin{dcases}
		g(z) - c \, \max_{b \geq 1} \, \frac{g(z) - \Ebb_0\bigl[\rho\bigl(z \frac{p_1(X)}{p_0(X)}, b - 1\bigr) \bigr]}{b}, & c < 1, \\
		\min\biggl\{\Ebb_0\biggl[\rho\biggl( z \frac{p_1(X)}{p_0(X)}, c - 1\biggr)\biggr] \,,\, g(z) - c \, \max_{b \geq c} \frac{g(z) - \Ebb_0\bigl[\rho\bigl(z \frac{p_1(X)}{p_0(X)}, b - 1\bigr) \bigr]}{b} \biggr\}, & c \geq 1.
	\end{dcases}
	\label{eq:rho_limit_recursive}
\end{equation*}
Since
\begin{equation*}
	g(z) - c \, \frac{g(z) - \Ebb_0\bigl[\rho\bigl(z \frac{p_1(X)}{p_0(X)}, b - 1\bigr) \bigr]}{b} \biggr|_{b=c} = \Ebb_0\biggl[\rho\biggl(z \frac{p_1(X)}{p_0(X)}, c - 1\biggr) \biggr],
\end{equation*}
the outer minimum in the case $c \geq 1$ can be omitted and $\rho$ can be written as
\begin{equation*}
	\rho(z, c) = \begin{dcases}
		g(z) - c \, \max_{b \geq 1} \, \frac{g(z) - \Ebb_0\bigl[\rho\bigl(z \frac{p_1(X)}{p_0(X)}, b - 1\bigr) \bigr]}{b}, & c < 1, \\
		g(z) - c \, \max_{b \geq c} \frac{g(z) - \Ebb_0\bigl[\rho\bigl(z \frac{p_1(X)}{p_0(X)}, b - 1\bigr) \bigr]}{b}, & c \geq 1.
	\end{dcases}
\end{equation*}
Combining both cases and rearranging the terms yields \eqref{eq:rho_bellman}. This completes the proof.

\section{Proof of Lemma~\ref{lm:rho_properties}} 
\label{apx:rho_properties}											 

The first statement in the lemma follows directly from the definition of $\rho$. In order to show convexity, let $c_1, c_2 \geq 0$ with $c_2 > c_1$, and let $\psi_1^* \in \Ical_{c_1}^\infty$ and $\psi_2^* \in \Ical_{c_2}^\infty$ be optimal in the sense of \eqref{eq:nonpa_kiwei_stopping} for $c = c _1$ and $c = c_2$, respectively. Now, consider the mixed stopping policy
\begin{equation}
	\psi^{**} = \gamma \psi_1^* + (1-\gamma) \psi_2^*,
	\label{eq:mixed_policy}
\end{equation} 
where $\gamma \in [0,1]$. \eqref{eq:mixed_policy} has to be read in the sense that a test under $\psi^{**}$ uses policy $\psi_1^*$ with probability $\gamma$ and policy $\psi_2^*$ with probability $(1 - \gamma)$. By construction, it holds that
\begin{equation*}
	\Ebb_{\psi^{**}}\bigl[ \tau(\bm{x}) \bigr] = \gamma \Ebb_{\psi_1^*}\bigl[ \tau(\bm{x}) \bigr] + (1 - \gamma) \Ebb_{\psi_2^*}\bigl[ \tau(\bm{x}) \bigr]
	\leq \gamma c_1 + (1-\gamma) c_2
\end{equation*}
for all $\bm{x} \in \Xcal^\Nbb$ so that $\psi^{**} \in \Ical_{\gamma c_1 + (1-\gamma) c_2}^\infty$. The error probabilities of the corresponding test are given by
\begin{align}
	\alpha(\psi^{**}) &= \gamma \alpha(\psi_1^*) + (1-\gamma) \alpha(\psi_2^*), \label{eq:alpha_mix} \\
	\beta(\psi^{**}) &= \gamma \beta(\psi_1^*) + (1-\gamma) \beta(\psi_2^*). \label{eq:beta_mix}
\end{align}
It follows that
\begin{align}
	\gamma \rho_k(z, c_1) + & (1 - \gamma) \rho_k(z, c_2) \notag \\
	&= \gamma \min_{\psi \in \Ical_{c_1}^\infty} \left( \alpha(\psi) + z \, \beta(\psi) \right) + (1 - \gamma) \min_{\psi \in \Ical_{c_2}^\infty} \left( \alpha(\psi) + z \, \beta(\psi) \right) \notag \\
	&= \gamma \left( \alpha(\psi_1^*) + z \, \beta(\psi_1^*) \right) + (1 - \gamma) \left( \alpha(\psi_2^*) + z \, \beta(\psi_2^*) \right) \label{eq:optimality_assumption}\\
	&= \gamma \alpha(\psi_1^*) + (1-\gamma) \alpha(\psi_2^*) + z \left( \gamma \beta(\psi_1^*) + (1-\gamma) \beta(\psi_2^*) \right) \notag \\
	&= \alpha(\psi^{**}) + z \, \beta(\psi^{**}) \label{eq:error_mixture}\\
	&\geq \min_{\psi \in \Ical_{\gamma c_1 + (1-\gamma) c_2}^\infty} \alpha(\psi) + z \, \beta(\psi,)\label{eq:mix_to_optimal} \\
	&\;= \rho(z, \gamma c_1 + (1-\gamma) c_2), \notag 
\end{align}
where \eqref{eq:optimality_assumption} holds since $\psi_1^*$ and $\psi_2^*$ are optimal, \eqref{eq:error_mixture} follows from \eqref{eq:alpha_mix} and \eqref{eq:beta_mix}, and \eqref{eq:mix_to_optimal} holds since $\psi^{**} \in \Ical_{\gamma c_1 + (1-\gamma) c_2}^\infty$.

In order to show that $\rho(z, \bullet)$ is piecewise linear, we define the following recursion:
\begin{equation}
	\tilde{\rho}_k(z, c) \coloneq g(z) - c \, \sup_{b \geq \max\{1,c\}} \frac{g(z) - \Ebb_0\bigl[\tilde{\rho}_{k-1}\bigl(z\frac{p_1(X)}{p_0(X)}, b - 1\bigr)\bigr]}{b},
	\label{eq:rho_tilde_k_recursion}
\end{equation}
with $\tilde{\rho}_0 \coloneq \rho_0$. It is not hard to show that for all $z,c \geq 0$ the sequence $\bigl( \tilde{\rho}_k(z, c) \bigr)_{k \geq 0}$ is non-increasing and converges to the same limit as $\rho_k(z, c)$, that is,
\begin{equation*}
	\lim_{k \to \infty} \tilde{\rho}_k(z, c) = \lim_{k \to \infty} \rho_k(z, c) = \rho(z, c).
\end{equation*}
The statement in the lemma can now be proven by induction. Assume that for some integer $k \geq 0$ the function $\tilde{\rho}_k(z, \bullet)$ is piecewise linear with breakpoints in $\Nbb_{\geq 0}$. In turn, the function
\begin{equation*}
	\Ebb_0\biggl[\tilde{\rho}_{k-1}\biggl(z\frac{p_1(X)}{p_0(X)}, \bullet \biggr)\biggr]
\end{equation*}
is piecewise linear with breakpoints in $\Nbb_{\geq 0}$. Since the ratio of two positive affine functions is monotonic, the argument of the supremum in \eqref{eq:rho_tilde_k_recursion} is monotonic on all intervals $[n, n+1)$ with $n \in \Nbb_{\geq 0}$. Consequently, the maximum on the right-hand side of \eqref{eq:rho_tilde_k_recursion} is either attained at $b_k^* = c$ or at some breakpoint $m \in \Nbb_{\geq 0}$. 

For $b_k^* = c$ it holds that
\begin{equation*}
	\tilde{\rho}_k(z, c) = \Ebb_0\biggl[\tilde{\rho}_{k-1}\biggl(z\frac{p_1(X)}{p_0(X)}, c - 1\biggr)\biggr],
\end{equation*}
which is piecewise linear with breakpoints in $\Nbb_{\geq 0}$ by assumption. For $b_k^* = m$ it holds that 
\begin{equation*}
	\tilde{\rho}_k(z, c) = g(z) - c \, \frac{g(z) - \Ebb_0\bigl[\tilde{\rho}_{k-1}\bigl(z\frac{p_1(X)}{p_0(X)}, m - 1\bigr)\bigr]}{m},
\end{equation*}
which is affine in $c$. This completes the induction step. The induction hypothesis holds true by Lemma~\ref{lm:rho_0}. The statement in the lemma now follows from the fact that pointwise limits of affine functions are affine\footnote{A function is affine if and only if it is convex and concave. Since pointwise limits of finite convex/concave functions are convex/concave \cite[Theorem~10.8]{Rockafellar1970}, pointwise limits of affine functions are affine.} and that $\tilde{\rho}_k(z, \bullet)$ is affine on all intervals $[n, n+1]$, $n \in \Nbb_{\geq 0}$. This completes the proof.

\section{Proof of Theorem~\ref{th:optimal_stopping}} 
\label{apx:optimal_stopping}												 

To prove the theorem, it suffices to show that
\begin{equation}
	\max\{ m^*(z) \,,\, c\} = \argmax_{b \geq \max\{1,c\}} \frac{g(z) - \Ebb_0\bigl[\rho\bigl(z\frac{p_1(X)}{p_0(X)}, b - 1\bigr)\bigr]}{b}.
	\label{eq:b_opt}
\end{equation}
For $c < 1$, it follows from the convexity and piecewise linearity of $\rho(z, \bullet)$ that the maximum in \eqref{eq:b_opt} is attained at a positive integer, that is,
\begin{align*}
	\argmax_{b \geq 1} \frac{g(z) - \Ebb_0\bigl[\rho\bigl(z\frac{p_1(X)}{p_0(X)}, b - 1\bigr)\bigr]}{b} &= \argmax_{m \in \Nbb_{\geq 1}} \frac{g(z) - \Ebb_0\bigl[\rho\bigl(z\frac{p_1(X)}{p_0(X)}, m - 1\bigr)\bigr]}{m} \\
	&= m^*(z) = \max\{ m^*(z) \,,\, c\}.
\end{align*}
For $c \geq 1$, we distinguish two cases. First, assume that $c \leq  m^*(z)$. In this case
\begin{equation*}
	\argmax_{b \geq c} \frac{g(z) - \Ebb_0\bigl[\rho\bigl(z\frac{p_1(X)}{p_0(X)}, b - 1\bigr)\bigr]}{b} 
	= m^*(z) 
	= \max\{ m^*(z) \,,\, c\}.
\end{equation*}
Now, assume that $c > m^*(z)$. Since $\rho(z, \bullet)$ is convex, it is non-decreasing on $\Rbb_{\geq m^*(z)}$. Thus,
\begin{equation*}
	\argmax_{b \geq c} \frac{g(z) - \Ebb_0\bigl[\rho\bigl(z\frac{p_1(X)}{p_0(X)}, b - 1\bigr)\bigr]}{b}
	= c 
	= \max\{ m^*(z) \,,\, c\} .
\end{equation*}
This completes the proof.

\section{Proof of \eqref{EQ:BOPT_LB0}} 
\label{apx:proof_lb_approx}						 

Let
\begin{equation*}
	u(b) \coloneq -\frac{\log z + \mu_0 b}{\sigma_0\sqrt{b}}.
\end{equation*}
Differentiating the objective function in \eqref{EQ:BOPT_LB0} with respect to $b$ gives the first-order optimality condition
\begin{equation}
	\frac{\phi(u(b))}{\Phi(u(b))} = \frac{2\sigma_0\sqrt{b}}{\log z - \mu_0 b},
	\label{eq:foc}
\end{equation}
where $\phi$ denotes the probability density function (PDF) of the standard normal distribution. Let the smallest solution of \eqref{eq:foc} be denoted by $b'$, so that $b'$ is a lower bound on the maximizer. First, assume that $b' > 1$. In this case, it holds that
\begin{equation*}
	\frac{\phi(u(1))}{\Phi(u(1))} > \frac{2\sigma_0}{\log z - \mu_0}.
\end{equation*}
Since the left-hand side of \eqref{eq:foc} is decreasing in $b$, this implies that it crosses the right-hand side from above at $b = b'$. We now further bound $b'$ from below by introducing two successive under-approximations.

The Mills ratio bound \cite[Chapter~VII, Lemma~2]{Feller1968} gives
\begin{equation}
	\frac{\phi(u)}{\Phi(u)} > -u.
	\label{eq:mills_bound}
\end{equation}
Substituting \eqref{eq:mills_bound} back into \eqref{eq:foc} and cross-multiplying yields the quadratic equation
\begin{equation}
	\mu_0^2 b^2 + 2\sigma_0^2 b - (\log z)^2 = 0
	\label{eq:foc_mod}
\end{equation}
with unique positive root
\begin{equation*}
	b'' = \frac{-\sigma_0^2+\sqrt{\sigma_0^4+\mu_0^2(\log z)^2}}{\mu_0^2}.
\end{equation*}
Since \eqref{eq:mills_bound} is a lower bound on the left-hand side of \eqref{eq:foc} and the latter crosses the right-hand side from above, it holds that $b'' < b'$. Finally, since $\sigma_0^4 > 0$ and $\mu_0 < 0$ we have $\sqrt{\sigma_0^4+\mu_0^2(\log z)^2} > -\mu_0\log z$, so that
\begin{equation*}
	b'' > b''' = \frac{-\sigma_0^2-\mu_0\log z}{\mu_0^2} = -\frac{\log z}{\mu_0}-\frac{\sigma_0^2}{\mu_0^2}.
\end{equation*}
Chaining the inequalities yields $b' \geq b'' \geq b'''$ for $b' > 1$. Finally, using the Mills ratio bound, it is not hard to show that $b' = 1$ implies $b''' < 1$. This completes the proof.

\printbibliography

\end{document}